\documentclass{article}

\usepackage[toc, page]{appendix} 

\usepackage{geometry}
\usepackage{amsmath}
\usepackage{amssymb}
\usepackage{amsthm}

\usepackage{subfigure}
\usepackage{float}
\usepackage{fancyhdr}
\usepackage{graphicx}
\usepackage[colorlinks,            linkcolor=black,            anchorcolor=black,            citecolor=black            ]{hyperref}
\usepackage{mathtools}

\usepackage{epstopdf}

\numberwithin{equation}{section}

\graphicspath{ {./images/} }
\newtheorem{proposition}{Proposition}[section]
\newtheorem{theorem}{Theorem}[section]
\newtheorem{lemma}{Lemma}[section]

\newtheorem{definition}{Definition}

\newtheorem{remark}{Remark}
\newtheorem{assumption}{Assumption}

\DeclareMathOperator\arccosh{arccosh}
\newcommand{\p}{\partial}
\newcommand{\Div}{\nabla\cdot}

\title{Tumor growth with a necrotic core as an obstacle problem in pressure}

\author{Xu'an Dou\thanks{Beijing International Center for Mathematical Research, Peking University, Beijing, 100871, China (dxa@pku.edu.cn)}\quad\quad\quad\quad\quad Chengfeng Shen\thanks{School of Mathematical Sciences, Peking University, Beijing, 100871, China (2201110049@pku.edu.cn)}\quad\quad\quad\quad\quad   Zhennan Zhou\thanks{Beijing International Center for Mathematical Research, Peking University, Beijing, 100871, China (zhennan@bicmr.pku.edu.cn).}}
\begin{document}
\maketitle

\begin{abstract}

Motivated by the incompressible limit of a cell density model, we propose a free boundary tumor growth model where the pressure satisfies an obstacle problem on an evolving domain $\Omega(t)$, and the coincidence set $\Lambda(t)$ captures the emerging necrotic core. We contribute to the analytical characterization of the solution structure in the following two aspects. By deriving a semi-analytical solution and studying its dynamical behavior, we obtain quantitative transitional properties of the solution separating phases in the development of necrotic cores and establish its long time limit with the traveling wave solutions. Also, we prove the existence of traveling wave solutions incorporating non-zero outer densities outside the tumor bulk, provided that the size of the outer density is below a threshold. 
 
\end{abstract}



\section{Introduction}

\subsection{Impressible limit in a simplest setting}

Mechanical-based tumor growth model has received growing attention from a wide spectrum of interests, from theoretical mathematical analysis to parameter estimations against experimental data (see e.g. \cite{perthame2014heleAsym,ahelemellet2017hele,guillen2020heleshaw,falcó2023quantifying}). It often involves two key factors: cell movement driven by a mechanical pressure and cell growth or death regulated by various factors, including nutrient concentration. Mathematically it leads to nonlinear (coupled) partial differential equations, often with degenerate diffusion.

We first describe this kind of model in a simple setting \cite{perthame2014heleAsym}, before coming to the particular scenarios interested in this work. The tumor is described by a cell density $n_{\gamma}(x,t)$, with $x\in\mathbb{R}^d$ and $t\geq 0$ denoting space and time variables respectively. Here the subscript $\gamma$ indicates the dependence on a model parameter $\gamma>1$ to be specified in \eqref{con-relation}. The cell density is driven by the following equation
\begin{equation}\label{eq:compressible}
    \p_{t}n_{\gamma}-\Div(n_{\gamma}\nabla p_{\gamma})=G(c_{\gamma})n_{\gamma},\quad x\in\mathbb{R}^d,\,t>0.
\end{equation} Here the advection term $-\Div(n_{\gamma}\nabla p_{\gamma})$ captures the mechanical movement of cells. It is assumed that cells are driven by the velocity field given by the negative gradient of the pressure
\begin{equation}\label{Darcy-compressible}
    v_{\gamma}=-\nabla p_{\gamma}.
\end{equation} This relation \eqref{Darcy-compressible} is also known as the Darcy's law. And the pressure $p_{\gamma}$ relates to the cell density via a constitutive relation given by
\begin{equation}\label{con-relation}
    p_{\gamma}=\frac{\gamma}{\gamma-1}(n_{\gamma})^{\gamma-1}.
\end{equation} Here, the parameter $\gamma>1$ measures the stiffness of the pressure.

The cells are also subject to growth or death with a rate $G(c_{\gamma})$, which gives the term $G(c_{\gamma})n_{\gamma}$ on the right hand side of \eqref{eq:compressible}. Here $G(\cdot)$ is a given function $[0,+\infty)\to\mathbb{R}$, and $c_{\gamma}=c_{\gamma}(x,t)$ is the nutrient concentrations. Hence $G(c_{\gamma})=G(c_{\gamma}(x,t))$ depends on the space and time through the heterogeneity of $c_{\gamma}(x,t)$. The nutrient $c_{\gamma}(x,t)$ evolves by a process coupled with $n_{\gamma}$, which requires additional modeling. We shall fully describe the nutrient models of interest in Section \ref{Sec:2.2}.

Remarkably, the system \eqref{eq:compressible}-\eqref{con-relation} has a non-trivial limit when $\gamma\rightarrow+\infty$, which leads to a free boundary problem. It is called the incompressible limit and is first proved in \cite{perthame2014heleAsym}. For heuristic purposes, we present a formal derivation as follows.

Multiply \eqref{eq:compressible} by $\gamma n_{\gamma}^{\gamma-2}$, which is $\frac{d p_{\gamma}}{d n_{\gamma}}$ in view of \eqref{con-relation}, and we derive
\begin{equation}\label{eq:p-compressible}
    \p_t{p_{\gamma}}-|\nabla p_{\gamma}|^2=(\gamma-1)p_{\gamma}(\Delta p_{\gamma}+G(c_{\gamma})).
\end{equation} Assume that $p_{\gamma}$ and $c_{\gamma}$ have well-defined limits $p$ and $c$, and by taking $\gamma\rightarrow+\infty$ in \eqref{eq:p-compressible}, we formally obtain
\begin{equation}\label{complementary-relation}
    p(\Delta p+G(c))=0.
\end{equation}
 The relation \eqref{con-relation} becomes stiff when $\gamma$ goes to infinity. Precisely, in the limit $n\in[0,1]$ and the relation \eqref{con-relation} gives the Hele-Shaw graph
\begin{equation}\label{limit-hele-shaw-graph}
p\quad\begin{dcases}
    =0,\qquad 0\leq n<1,\\
    \in[0,+\infty),\quad n=1.
\end{dcases}
\end{equation} 

To derive a simplified geometric description, we consider patch solutions in the following form 
\begin{equation}\label{solu-char}
    n(x,t)=\mathbb{I}_{x\in \Omega(t)}.
\end{equation} In other words, for each time the density is an indicator function in space, where $\Omega(t)$ is an open set representing the tumor domain. In view of \eqref{limit-hele-shaw-graph} we assume
\begin{equation}
    \Omega(t)=\{x: n(x,t)>0\}=\{x: n(x,t)=1\}=\{x: p(x,t)>0\}.
\end{equation}
Then we deduce from \eqref{limit-hele-shaw-graph} that $p$ solves 
\begin{equation}\label{eq:p-Possion}
    \begin{dcases}
        -\Delta p=G(c),\quad \Omega(t),\\  
        p=0,\quad \p\Omega(t).\\
    \end{dcases}
\end{equation}
To describe how the domain $\Omega(t)$ evolves, following the Darcy's law \eqref{Darcy-compressible}, the boundary velocity in the outer normal direction $\boldsymbol{n}$ is given by
\begin{equation}\label{Darcy-incompresible}
     V_n=-\nabla p\cdot \boldsymbol{n},\quad x\in\p \Omega(t). 
\end{equation} System \eqref{eq:p-Possion}-\eqref{Darcy-incompresible} gives a Hele-Shaw type free boundary problem, which can be complemented by an initial data $\Omega(t)$ and some dynamics of the nutrient $c$.

This incompressible limit
links the cell density model \eqref{eq:compressible}-\eqref{con-relation} to the free boundary model \eqref{eq:p-Possion}-\eqref{Darcy-incompresible}, which itself belongs to another popular class of tumor growth models \cite{roose2007mathematical,friedman2007mathematical,lowengrub2010nonlinear}. Since the seminal work \cite{perthame2014heleAsym} there have been many studies on the incompressible limits in various settings e.g. \cite{bubba2019-helelimit2,david:hal-02515263,guillen2020heleshaw,dou2023tumor}. In particular, the geometric description  \eqref{eq:p-Possion}-\eqref{Darcy-incompresible} can be rigorously justified \cite{ahelemellet2017hele,kim2018porous} which relies on proving that the solution is of the patch form \eqref{solu-char}. Usually, two assumptions are needed: i) The growth rate is nonnegative $G\geq 0$ and ii) the initial data is a characteristic function $n(x,0)=\mathbb{I}_{x\in\Omega(t)}$. 

It is worth emphasizing that those two assumptions are mathematical idealizations, which lead to limitations to the application of these models. On the one hand, as tumors grow beyond the nutrient supply, cells in the tumor interior starve, become hypoxic, and die, leading to the formation of a necrotic core \cite{chaplain1993development,bull2020mathematical}. Accounting for a negative growth rate is necessary because the presence of the necrotic region alters tumor morphology, mechanics, transport, cell proliferation kinetics, and treatment response through complex feedback mechanisms. On the other hand, while cell density drops off substantially with distance from the tumor border, empirical evidence and scientific studies \cite{jass1987new,kienast2010real}  support a decaying but non-zero tumor cell distribution outside the primary tumor.

However, once the above two assumptions are relaxed, the free boundary problem \eqref{eq:p-Possion} and \eqref{Darcy-incompresible} may fail to be the correct description of the limit system. This calls for free boundary problems with new mechanisms, which we motivate and introduce as follows. 
\subsection{Beyond simplest setting}

\subsubsection{Necrotic core as an obstacle problem in pressure}
To model the necrosis of tumor cells, e.g. due to insufficient nutrients, it is natural to allow negative $G$. However, in this case, the incompressible limit of \eqref{eq:compressible} is not \eqref{eq:p-Possion}-\eqref{Darcy-incompresible} in the current form, as pointed out in \cite[Section 7.3]{perthame2015some}. One way to see it is to consider an example when $G$ is positive only near the boundary of $\p\Omega(t)$ and negative in the core of $\Omega(t)$. In that case, directly solving equation \eqref{eq:p-Possion} may result in \textit{negative} pressure, see Figure \ref{fig:1} (solid line). This contradicts the non-negativity of $p$, which shall be inherited from the compressible model \eqref{con-relation}
\begin{equation}
    p=\lim_{\gamma\rightarrow+\infty}p_{\gamma}\geq 0.
\end{equation}
\begin{figure}[htbp]
        \centering
        \includegraphics[width=0.5\linewidth]{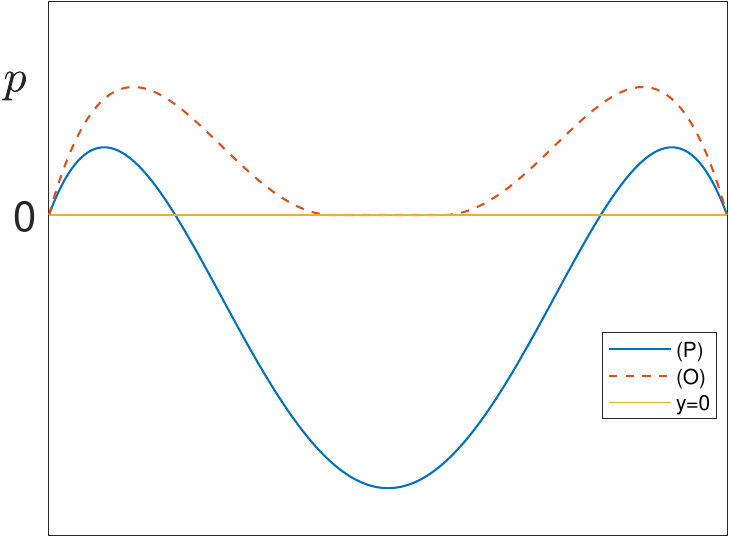}
        \caption{Comparison of the solutions to the Possion equation \eqref{eq:p-Possion} (solid line, with label $(P)$) and to the obstacle problem (dashed line, with label (O)). We see solving the Possion equation can leads to negative pressure while for the obstacle problem the solution is non-negative, and touches zero in the center.} 
        \label{fig:1}
\end{figure}
Another observation is that when $G$ takes negative values, a patch solution \eqref{solu-char}  may not persist, and there may appear a region where the density is unsaturated, i.e. $n(x,t)<1$.  

Biologically, the unsaturated region corresponds to a necrotic core. The cells at the tumor boundary have better accessibility to nutrients, while in the core the nutrient supply is not sufficient. Therefore while near the boundary tumor cells may grow and invade, in the core of the tumor cells may die, leading to the formation of a necrotic core. 

Hence, it is not only mathematically interesting to characterize the incompressible limit allowing $G<0$, but also scientifically meaningful as it can give a description of the necrotic core in this modeling framework. In this direction, the traveling wave solution of the limit model is studied in \cite{perthame2014traveling} and the incompressible limit has been rigorously derived and studied in \cite{david:hal-02515263,guillen2020heleshaw}. Notably, for the characterization of $p$, \cite{guillen2020heleshaw} draws an interesting link to {obstacle problems}. 

Inspired by previous works, to better understand the necrotic core, here we propose and study a free boundary model directly. Our model consists of four components: the evolving domain $\Omega(t)$ for the tumor bulk, the pressure $p$, the nutrient concentration $c$, and the cell density $n$. The key distinctive feature lies in the characterization of $p$. At each time $t$ with given domain $\Omega(t)$ and $c(x,t)$, instead of solving the Possion equation \eqref{eq:p-Possion}, $p$ solves an \textit{obstacle problem}, that is 
\begin{equation}\label{minimization-ob}
            p(\cdot,t)=\arg\min_{u \in E} \left(\int_{\Omega(t)}\bigl(\frac{1}{2}|\nabla u|^2-G(c(x,t))u\bigr) dx\right), \quad E = \{ u: {u\in H^1_0(\Omega(t)),\, u\geq 0\ \text{a.e.}} \}.
\end{equation}Then by definition $p\geq 0$, as it is the minimizer among a family of non-negative functions. Note that without the non-negative constraint on $u$, the obstacle problem \eqref{minimization-ob} becomes the variational characterization of the Possion equation \eqref{eq:p-Possion}. See Figure \ref{fig:1} for a comparison of the profiles of two solutions.

Given a solution $p$ to \eqref{minimization-ob}, we can divide $\Omega(t)$ into two parts as 
\begin{equation}\label{def:necrotic}
\begin{aligned}
        \Omega(t)&=\{x\in \Omega(t), p(x,t)=0\} \cup \{x\in \Omega(t), p(x,t)>0\}\\
        &=:\Lambda(t)\cup N(t).
\end{aligned}
\end{equation} In the classical theory of obstacle problems (see e.g. \cite[Chapter 1]{MR1009785}), $\Lambda(t)$ is called the coincidence set, as in it $p$ coincides with zero. And $N(t)$ is called the non-coincidence set.

Here we may name $\Lambda(t)$ a (mathematical) necrotic core. As by the limit relation \eqref{limit-hele-shaw-graph}, we have $n=1$ in $N(t)$ but $n\leq 1$ in $\Lambda(t)$. Moreover, we shall see later $G(c)\leq 0$ in the interior of $\Lambda(t)$ (e.g. \eqref{ineq-ob}). These indicate that in $\Lambda(t)$ cells are undergoing necrosis due to the lack of nutrients and the density is unsaturated.  The necrotic core $\Lambda(t)$ is a key concept emerging from formulation \eqref{minimization-ob}, which will be further investigated in this paper. See Figure \ref{fig:2} for an illustration.  
\begin{figure}[htbp]
    \centering
    \includegraphics[width=0.5\textwidth]{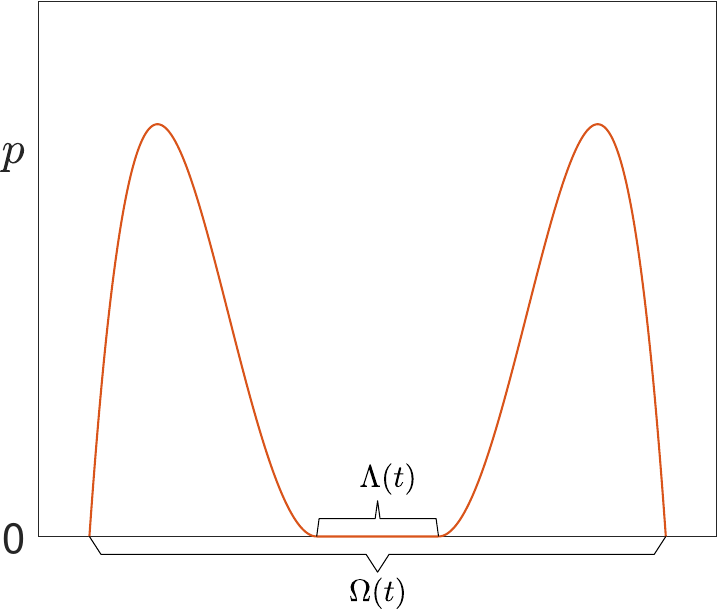}
    \caption{An illustration for $\Lambda(t)$. Pressure $p$ solves an obstacle problem on tumor bulk $\Omega(t)$ and the coincidence set $\{x\in\Omega(t):p=0\}$ gives the necrotic core $\Lambda(t)$. }
    \label{fig:2}
\end{figure}

It is worth noting that the formulation {proposed} here is akin to the characterization of the limit problem already {justified} in \cite{guillen2020heleshaw}, both say $p$ satisfies an obstacle problem. Briefly speaking, the difference is that here we prescribe a tumor bulk domain $\Omega(t)$ on which the obstacle problem is solved. It is expected that the formulation proposed here can be viewed as a special scenario in the general result of \cite{guillen2020heleshaw}, under certain assumptions. Besides, our formulation also morally agrees with the treatment in \cite{perthame2014traveling}, although in that work the necrotic core was not explicitly connected to an obstacle problem.  More detailed comments will be given later in Section 2 after the complete model is presented.

 Instead of being derived from an incompressible limit of \eqref{eq:compressible}, free boundary models for the necrotic core have also been proposed directly, to which many studies have been devoted \cite{byrne1996growth,cui2001analysis,cui2006formationNecrotic,HAO2012694}, see in particular recent mathematical analysis \cite{cui2019analysisNecrotic,wu2021analysis}. These models have similar but different dynamics compared to ours. In particular, in these models an obstacle problem also appears, but it is {for \textit{nutrient} $c$, rather than for \textit{pressure} $p$}.  A more detailed discussion on this will be made in Section \ref{sec:conclusion}. 

\subsubsection{Outer Density}

To complete the model, in addition to the description of $p$ \eqref{minimization-ob} we need to determine how the other three components---the domain $\Omega(t)$, the nutrient concentration $c$, and the density $n$ evolve. 

In general, the density function takes the following form
\begin{equation}\label{outer-density-general}    n(x,t)= n^{\text{in}}(x,t) \mathbb{I}_{x\in\Omega(t)}+n^{\text{out}}(x,t),
\end{equation}
where $\forall t>0$, $n^{\text{out}}(x,t) \in [0,1)$ is supported in $\mathbb R^d \backslash \Omega(t)$, and is referred to as the outer density while $n^{\text{in}}(x,t) \in [0,1]$ is the density inside the tumor bulk. For patch solutions, we simply have $n^{\text{out}} \equiv 0$ and 
$n^{\text{in}} \equiv 1$. We remark that when the growth rate $G$ may take negative values, it is difficult to describe $n^{\text{out}}$ in an explicit way. See Section \ref{Sec:2.3} for an illustrating example. 

The presence of the outer density will also influence the boundary propagation speed of the tumor bulk region $\Omega(t)$. In this scenario, Darcy's law still applies to derive its boundary moving speed. While in the patch solution case, the density jumps to $0$ across the tumor bulk boundary  $\p\Omega(t)$, in general, the density jumps to a finite value $n^{\text{out}}(x,t) \in [0,1)$ across $\p\Omega(t)$, and the velocity law \eqref{Darcy-incompresible} needs adjustment for account for the effect of the outer density. The precise modified law has already been established in \cite{ahelemellet2017hele,kim2018porous,guillen2020heleshaw}. We incorporate the velocity law as part of the model and will present it in detail in Section 2 along with the dynamics for $c$ and $n$.

\subsection{Main results}

To summarize, we consider a free boundary model motivated by the incompressible limit of the cell density model \eqref{eq:compressible}. The key feature is that at each time the pressure $p$ solves an obstacle problem on an evolving domain $\Omega(t)$, which gives the necrotic core $\Lambda(t)$ as the coincidence set. The complete model also involves the dynamics of nutrient $c$ and cell density $n$, which are to be elaborated on in Section \ref{sec:2}.

Our work explores two aspects of the proposed tumor growth model. First, we derive a semi-analytical solution without outer density, elucidating the evolving solution profile and long-term limits. Second, we prove the existence of traveling wave solutions accounting for non-zero densities outside the tumor.

Together, these contributions provide an analytical characterization of two important yet insufficiently examined challenges in mathematical tumor growth models: necrotic core formation and outer density evolution. The semi-analytical solution offers unique insights into tumor morphology changes over time. The traveling wave results mathematically demonstrate the impact of considering more realistic density profiles. 

The rest of the paper is arranged as follows. We present the proposed model in its complete form in Section \ref{sec:2}. The semi-analytical solution for the formation of the necrotic core is derived and studied in Section \ref{sec:analytical}. And the existence of the traveling wave solution with outer density is studied in Section \ref{sec:4}. Finally, conclusions and discussions are given in Section \ref{sec:conclusion}.

\section{Model introduction}\label{sec:2}

In this section, we present the complete model for tumor growth, where the tumor bulk is described by an evolving domain $\Omega(t)$ (possibly containing a necrotic core). It is coupled with three functions: the pressure $p$, the nutrient concentration $c$, and the cell density $n$.

We begin by revisiting the obstacle problem formulation for pressure $p$ in Section 2.1, with a preliminary analysis to elucidate the solution structure. Subsequent sections give the dynamics of the other components: Section 2.2 focuses on nutrient $c$; Section 2.3 delves into the dynamics of cell density $n$; and Section 2.4 describes how the tumor bulk $\Omega(t)$ evolves. Finally, Section 2.5 offers a summary and discussions.

\subsection{Obstacle problem for pressure} \label{Sec:2.1}
As present in the introduction, for each time $t$, the pressure $p(\cdot,t)$ solves the obstacle problem \eqref{minimization-ob}, on domain $\Omega(t)$ with the source term $G(c(\cdot,t))$. Let us give some preliminary analysis based on the classical theory of obstacle problems (see e.g. \cite{MR1009785,fernandez2023regularity}).

\paragraph{Preliminary Analysis}

Firstly, \eqref{minimization-ob} has a unique solution $p\geq0$ with certain regularity, given a regular (e.g. bounded) source $G(c(\cdot,t))$ and a regular domain $\Omega(t)$. 

And following the terminology from the theory of the obstacle problem, the domain $\Omega(t)$ is divided into two parts: the coincidence set $\Lambda(t)$ and the non-coincidence set $N(t)$, defined via \eqref{def:necrotic}
\begin{equation}\label{def:necrotic-2}
\begin{aligned}
        \Lambda(t)&:=\{x\in \Omega(t),\quad p(x,t)=0\},\\ N(t)&:=\{x\in \Omega(t),\quad p(x,t)>0\}.
\end{aligned}
\end{equation}
As discussed in the introduction, $\Lambda(t)$ is interpreted as the necrotic core and we might view $N(t)$ as the proliferating rim, i.e. the saturation region inside the tumor bulk. 

Also, by the classical theory we have a super-solution inequality
\begin{equation}\label{ineq-ob}
    \Delta p +G(c)\leq 0,\qquad x\in\Omega(t),
\end{equation} in the distributional sense. This in particular implies that $G\leq 0$ in the interior of $\Lambda(t)$, which indicates that cells are suffering from the lack of nutrients in this domain.

Another classical result on obstacle problem states
\begin{equation}
    \nabla p=0,\qquad x\in\p \Lambda(t)\cap \Omega(t),
\end{equation} which describes the local behavior of pressure near the boundary of the necrotic core. 

Finally, we recall
\begin{equation}\label{com-ob}
    p(\Delta p+ G(c))=0,\quad x\in \Omega(t),
\end{equation} as a basic property for obstacle problem \eqref{minimization-ob}. In the context of the incompressible limit, it is also known as the complementary relation, as derived in \eqref{eq:p-compressible}-\eqref{complementary-relation}. As a consequence, \eqref{com-ob} implies that $p$ solves the Possion equation $\Delta p+G(c)=0$ whenever $p>0.$ Indeed, we can recast the obstacle problem \eqref{eq:p-Possion} as 
\begin{equation}\label{ob-explicit}
        \begin{dcases}
            p\geq 0,\quad x\in\Omega(t),\\
            {\Lambda(t)=\{x\in\Omega(t):p(x,t)=0\}},\\
                        \nabla p=0,\quad x\in\partial\Lambda(t)\cap \Omega(t),\\
            -\Delta p=G(c),\quad x\in N(t)=\Omega(t)\backslash\Lambda(t),\\
            p=0,\quad x\in\partial\Omega(t).\\
        \end{dcases}
    \end{equation}  This more explicit formulation is useful for later calculations. We stress that in \eqref{ob-explicit} the necrotic core $\Lambda(t)$ is unknown a priori, but part of the solution implicitly determined by \eqref{ob-explicit}. 
    
    Note that $\Lambda(t)$ can be empty when the solution to the Poisson equation \eqref{eq:p-Possion} happens to be non-negative. This is always the case when $G(c(\cdot,t))>0$, which means the necrotic core does not form when the nutrients are sufficient.
\subsection{Dynamics of nutrients} \label{Sec:2.2}

To complete the model, first we supply the dynamics of nutrient $c$. There are two common settings \cite{perthame2014traveling,feng2023tumor}. Let us first look at the \textit{in vitro} one, given by
\begin{equation}\label{eq:model-c-in-vitro}
\begin{dcases}
-\Delta c+\psi(n)c=0  ,\quad x\in \Omega(t),\\
c=c_B  ,\quad x\notin\partial\Omega(t).\\
\end{dcases}
\end{equation}

In this setting, it is assumed that the nutrient is sufficient at a constant level $c_B>0$ as long as $x$ is not in the tumor bulk $\Omega(t)$. This in particular gives the boundary condition $c=c_B$ at $\p \Omega(t)$. Inside $\Omega(t)$ the nutrient is subject to diffusion and consumption by cells. It is assumed that nutrient diffusion is at a faster timescale therefore with a quasi-static approximation  \eqref{eq:model-c-in-vitro} is in an elliptic form. And $\psi(n)\geq 0$ is the consumption rate which depends on the cell density.

Another setting is \textit{in vivo}, given by 
\begin{equation}\label{eq:model-c-in-vivo}
\begin{dcases}
-\Delta c+\psi(n)c=0  ,\quad x\in \Omega(t),\\
-\Delta c+\psi(n)c=c_B-c  ,\quad x\notin \Omega(t),\\
\text{$c$ is $C^1$ at $\p\Omega(t)$ },\\
c\rightarrow c_B ,\quad \text{as}\quad |x|\rightarrow+\infty .\\
\end{dcases}
\end{equation} Unlike the \textit{in vitro} case \eqref{eq:model-c-in-vitro}, here in \eqref{eq:model-c-in-vivo} the nutrient $c$ solves an equation on the whole domain, with an extra source term $c_B-c$ outside $\Omega(t)$. The underlying model assumption is that the nutrient is supplied by the vasculature network, which is available only outside the tumor bulk.

\subsection{Dynamics of density} \label{Sec:2.3}

In this part, we focus on the dynamics of cell density $n$, which is supposed to be self-contained for the proposed model and consistent with the existing works on incompressible limits, in particular \cite{guillen2020heleshaw}. While it is not our intent to recapitulate these seminal findings in their most general form, our formulation can be viewed as a special case, under some assumptions.

First, recall the patch solution case, when initial density is a characteristic function i.e. $n_{\text{init}}(x)=\mathbb{I}_{x\in\Omega(0)}$, and the patch solution persists. That is, $n(\cdot,t)$ remains to be a characteristic function for each $t$ occupying the set $\Omega(t)$, resulting in
\begin{equation}
\begin{dcases}
        n(x,t)=1,\quad x\in\Omega(t),\\
    n(x,t)=0,\quad x\notin\Omega(t).
\end{dcases}
\end{equation} This patch solution can be maintained when $G$ is non-negative.

Now that we allow the growth rate $G$ to take negative values and the presence of outer density outside the tumor bulk $\Omega(t)$, the dynamics of the density becomes more complicated. We consider the initial data in the following  form
\begin{equation}\label{outer-density-initial-data}
    n_{\text{init}}(x)=\mathbb{I}_{x\in\Omega(0)}+n_{\text{init}}^{\text{out}}(x),
\end{equation} which is a characteristic function plus an ``outer density'' with
\begin{equation}    n_{\text{init}}^{\text{out}}(x)=0,\quad x\in \Omega(0),\qquad\qquad 0\leq n_{\text{init}}^{\text{out}}(x)<1\quad x\notin \Omega(0).
\end{equation} 
Here, $\Omega(0)$ is a bounded connected open set in $\mathbb R^d$, describing the initial tumor bulk region, and there may develop a necrotic core $\Lambda(t)$ inside $\Omega(t)$.

Given a specific time $t$, when $x\in N(t)$, by the Hele-Shaw graph \eqref{limit-hele-shaw-graph}, we immediately know that
\begin{equation}
    n(x,t)=1,\quad x\in N(t)=\Omega(t)\backslash \Lambda(t).
\end{equation}
In other words, in the proliferating rim, the density is always saturated. 

However, when $x\in \mathbb R^d \backslash N(t)$, $x$ is either located outside the tumor bulk or in the necrotic core, the pressure is zero and cells are only subject to local growth or death with rate $G$. Hence, we have 
\begin{equation} \label{eq:rate_eq}
    \frac{d}{dt}n(x,t)=G(c(x,t))n(x,t),\quad x\in \mathbb R^d \backslash N(t).
\end{equation}
In particular, as we discussed in Section \ref{Sec:2.1}, $G\leq0$ in the interior of the necrotic core $\Omega(t)$, and the density will continue to decrease. Whereas, it is likely that the nutrient is sufficient outside the tumor bulk, and the density will increase.

\paragraph{A heuristic example}

So far, we have described the density dynamics, but since we have two moving boundaries defining the tumor bulk and the necrotic core respectively, it is helpful to look at the density dynamics at a fixed spatial position $x$.  Assume initially, $x$ is located outside the tumor bulk, i.e. $x \notin \Omega(0)$,  $t_0$ is the first time the tumor occupies $x$, and $t_1 (>t_0)$ is the first time the necrotic core forms at $x$. More precisely, these times are defined as
\begin{equation}\label{def-t0t1}
    t_0=\inf\{t\geq 0,\, x\in \Omega(t)\},\qquad\qquad\qquad t_1=\inf\{t\geq 0,\, x\in \Lambda(t)\}.
\end{equation} 
Then the density dynamics at $x$  is given by
\begin{equation}\label{dynamics-n-fixed-x-outer}
\begin{dcases}
    n(x,0)=n_{\text{init}}(x),\\
    \frac{d}{dt}n(x,t)=G(c(x,t))n(x,t),\quad t\in(0,t_0),\\
    n(x,t)=1,\quad t\in(t_0,t_1),\\
    \frac{d}{dt}n(x,t)=G(c(x,t))n(x,t),\quad t\in(t_1,+\infty).
\end{dcases}
\end{equation} 

Suppose that the nutrient is sufficient outside the tumor bulk, and thus $G(c(x,t))>0$ for $x\notin \Omega(t)$. Hence in the time interval $(0,t_0)$ the density is increasing, until it jumps to $1$ at $t_0$, at which time the position $x$ starts to be occupied by the tumor bulk. Finally, at $t_1$, position $x$ becomes part of the necrotic core and the density starts to decrease.

Note that for the above dynamics to make sense we have implicitly assumed that: the density outside the tumor bulk is strictly less than $1$ before the given position becomes part of $\Omega(t)$. In other words, tumor cells can not form new saturated regions before being absorbed by the main bulk; the tumor bulk $\Omega(t)$ and the necrotic core $\Lambda(t)$ monotonically increase over time, which ensures that once position $x$ becomes part of $x\in \Lambda(t)$, it remains so for all $t>t_1$. But still, this example demonstrates that the local density can be uniquely determined given the propagation of the moving boundaries.   

We conclude this part by providing a numerical simulation of the tumor growth model with \textit{in vitro} nutrient dynamics. The model parameters are given by: $c_B=1,\bar{c}=0.9$,
$$
\psi(n)=n,\quad \mbox{and} \quad G(c)=\begin{cases}
    -2,&c<\bar{c};\\
    0.2(c-\bar{c})+1.2,& c\ge \bar{c}.
\end{cases}
$$
We solve the problem on the spatial domain $[0,3]$, where the initial density is given by
$$
n^0(x)=\begin{cases}
    1,&1.2\le x\le 1.8;\\
    0.1e^{-40(|x-1.5|-0.3)},& 0\le x<1.2 \text{ or }1.8<x\le 3.
\end{cases}
$$
The simulation results are shown in Figure \ref{fig:8}. We observe that the cell density is initialized as a patch with a decaying outer density, and as the tumor grows, an expanding necrotic core emerges at the center while the outer density persists. From the position $x=2$, we observe from the last picture of Figure \ref{fig:8} that it is initially located in the outer region and its local density grows continuously until the tumor bulk invades at $t=t_0 \approx 4$ and the density jumps immediately to $1$. As the evolution goes on, the position $x=2$ stays in the saturation region until it is occupied by the necrotic core at $t=t_1 \approx 21 $, and the density starts to drop continuously. 

  \begin{figure}[htbp]
    \centering
\includegraphics[width=0.4\textwidth]{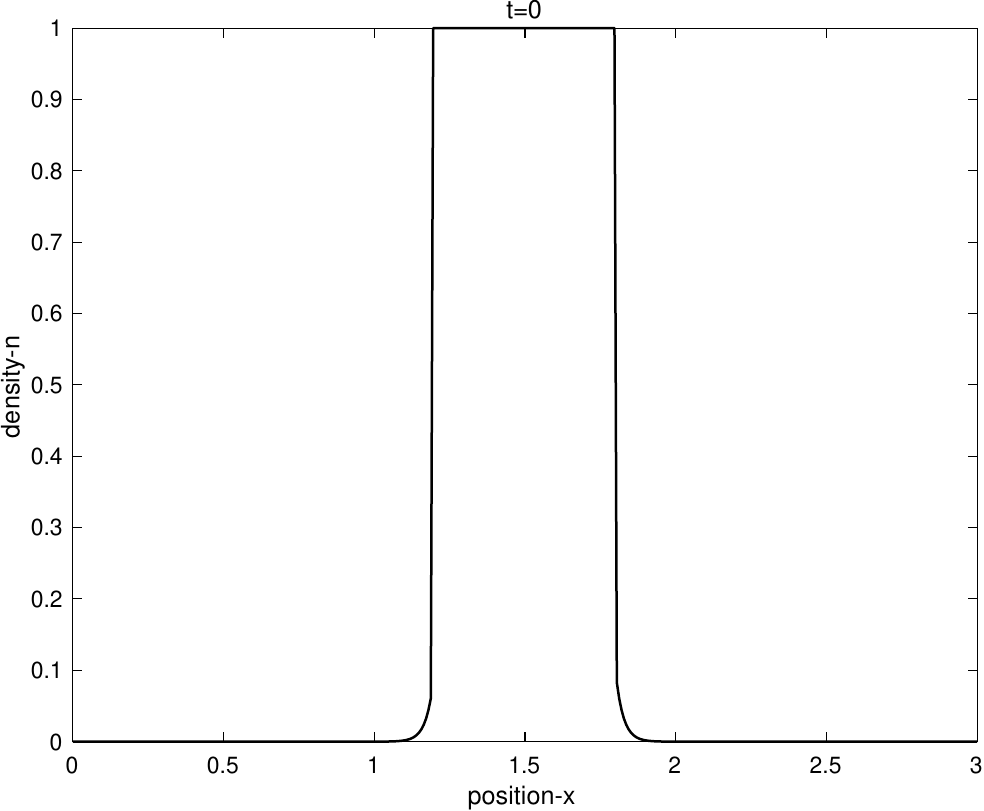} \quad
\includegraphics[width=0.4\textwidth]{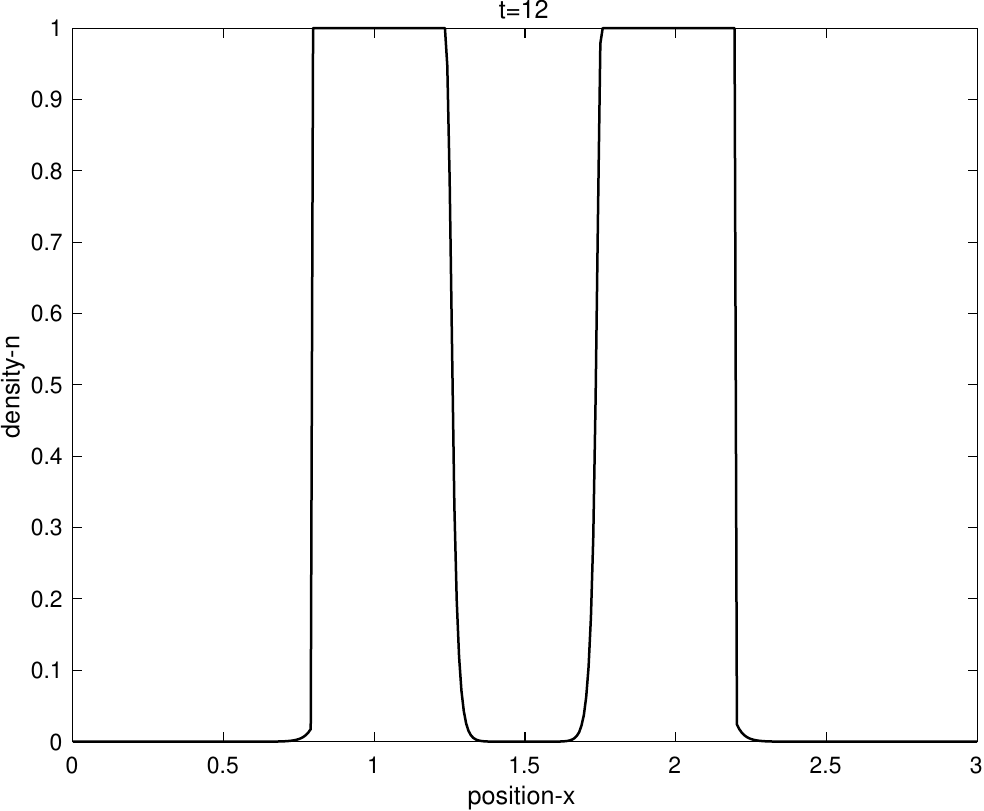} \\
\quad \includegraphics[width=0.4\textwidth]{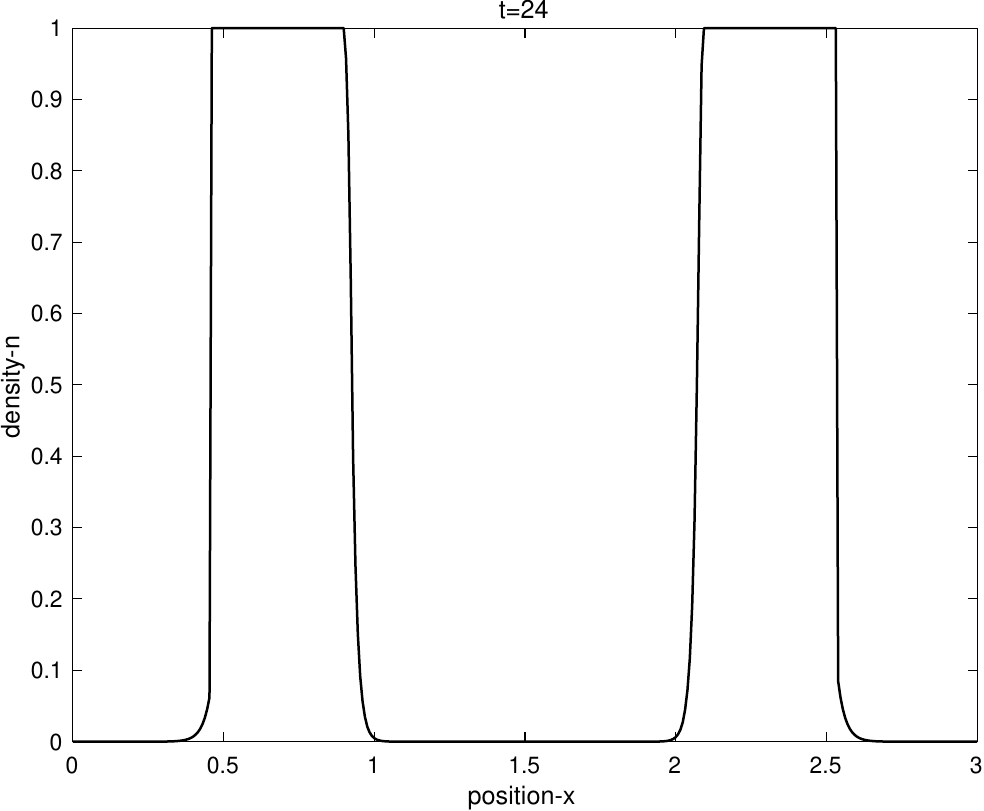}
\includegraphics[width=0.455\textwidth]{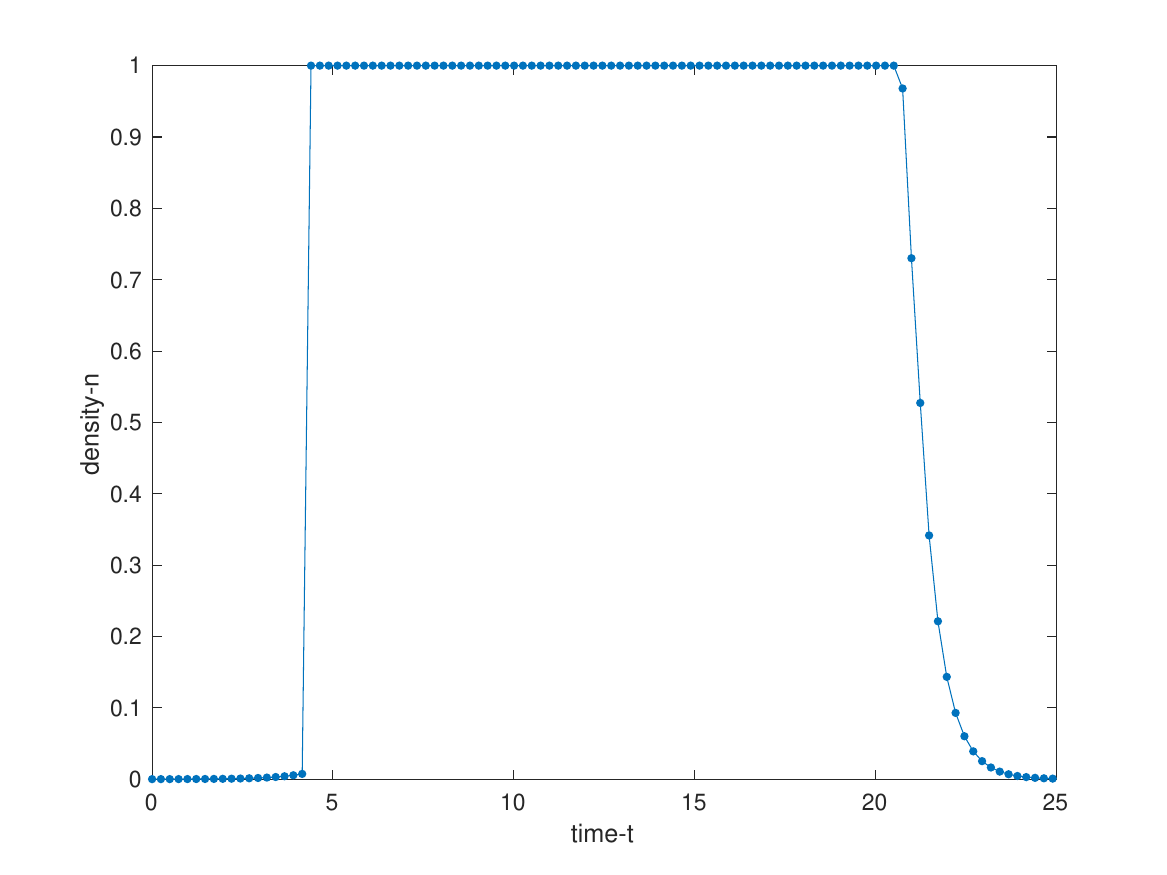}
    \caption{Numerical simulation of the tumor growth model. The first three plots are the cell density on $[0,3]$ at $t=0$, $t=12$ and $t=24$, respectively. The last picture is the time evolution of the local density at $x=2$. }
    \label{fig:8}
\end{figure}

\subsection{Dynamics of the tumor bulk: Boundary moving speed.}

In previous sections, we have discussed the dynamics of $p$, $c$ and $n$, treating $\Omega(t)$ as a family of given domains. Now we describe how $\Omega(t)$ evolves, which completes the model. Precisely, we give the boundary moving speed of $\p\Omega(t)$.

In the case without outer density (initially $n(x,0)=\mathbb{I}_{x\in\Omega(0)}$), the boundary moves according to the Darcy's law as \eqref{Darcy-incompresible}
\begin{equation}\label{Darcy-incompresible-revisited}
     V_n=-\nabla p\cdot \boldsymbol{n},\quad x\in\p \Omega(t). 
\end{equation}

In the presence of outer density (c.f. \eqref{outer-density-initial-data}), \eqref{Darcy-incompresible-revisited} is modified to

\begin{equation}\label{Darcy-incompresible-outer}
     V_n(1-n^{\text{out}}(x,t))=-\nabla p\cdot \boldsymbol{n},\quad x\in\p \Omega(t), 
\end{equation} where for $x\in \p\Omega(t)$ the outer density at $x$ is defined via
\begin{equation}
    n^{\text{out}}(x,t):=\lim_{y\rightarrow x, y\notin\Omega(t)}n(x,t),
\end{equation} which is assumed to be less than $1$. With a factor $1/(1-n^{\text{out}})$, \eqref{Darcy-incompresible-outer} means the tumor accelerates invasion when the outer density is positive.

\subsection{Summary and discussion}

We have described the joint dynamics of an evolving domain $\Omega(t)$ and three functions $p$, $c$, $n$: the pressure solves an obstacle problem at each time \eqref{minimization-ob}/\eqref{ob-explicit}; the nutrient concentration has two versions of dynamics, either \textit{in vitro} \eqref{eq:model-c-in-vitro}  or \textit{in vivo} \eqref{eq:model-c-in-vivo}; the density evolves according to \eqref{dynamics-n-fixed-x-outer}; and finally the boundary moving speed of $\Omega(t)$ is given by \eqref{Darcy-incompresible-outer}. 

We note that the whole space is divided into three parts: outside tumor bulk, proliferation rim, and the necrotic core:
\begin{equation}
    \mathbb{R}^d=(\mathbb{R}^d\backslash\Omega(t))\cup N(t)\cup \Lambda(t),
\end{equation} where the latter two parts consist of the tumor bulk $\Omega(t)$, and the necrotic core $\Lambda(t)$ can be empty.

The density $n$ can also be decomposed into three parts accordingly
\begin{equation}
    n(x,t)= \mathbb{I}_{x\in \mathbb R^d \backslash \Omega(t)} n^{\text{out}}(x,t)+\mathbb{I}_{x\in N(t)}+\mathbb{I}_{x\in \Lambda (t)}n^{\text{core}}(x,t),
\end{equation} where the density is saturated on $N(t)$. In the outer part and the necrotic core part, it undergoes growth or death driven by the rate equation \eqref{eq:rate_eq}.

For pressure $p$, by Definition \eqref{def:necrotic-2} we know $p>0$ on $N(t)$ and $p=0$ on $\Lambda(t)$. It is also natural to set $p=0$ outside $\Omega(t)$ as a zero extension. Hence to be consistent with the relation \eqref{limit-hele-shaw-graph}, we restrict ourselves to the case $n^{\text{out}}(x,t)<1$ -- the density can jump to $1$ only when $x$ is invaded by $\Omega(t)$. In other words, cells can not form new ``islands'' in isolation, without being incorporated into by the main tumor bulk. We stress that this restriction is necessary to focus on the case when $\Omega(t)$ has a simple geometry, e.g. a sphere, which allows us to completely describe its dynamics via the boundary moving speed \eqref{Darcy-incompresible-outer}.

The domain $\Omega(t)$ evolves according to the boundary moving speed \eqref{Darcy-incompresible-outer} of $\p\Omega(t)$. While there might also be an inner boundary $\p\Lambda(t)$ for the necrotic core, we do not specify its boundary moving speed. Rather, $\Lambda(t)$ is determined once $\Omega(t)$ is fixed, as it is defined as the coincidence set \eqref{def:necrotic-2} in the obstacle problem for $p$ \eqref{minimization-ob}. While in general the coincidence set in an obstacle problem can be complicated, here we assume that $\Lambda(t)$ (if non-empty) is indeed like a core in $\Omega(t)$, in particular $\p\Lambda(t)\subset\Omega(t)$. This is the case in later computations. 

In a related work \cite{guillen2020heleshaw}, the pressure is proved to satisfy an obstacle problem but on the $\{x:n(x,t)=1\}$. Here we enlarge the domain on which the obstacle problem is solved, and the enlarged domain $\Omega(t)$ may contain a necrotic core where $n$ can be strictly less than $1$. This enlargement of the domain can be justified in simple cases, e.g. when $\Omega(t)$ is a sphere and the nutrient concentration $c$ is radially symmetric and increases with respect to the radius. The justification for more general cases is beyond the scope of this paper. Whereas, we would like to point out that such a relaxation of the domain allows us to partially decouple the joint dynamics, which facilitates the computation of the analytical solution in the next section.

\section{Semi-analytical studies on the necrotic core}\label{sec:analytical}

In this section, we study the model proposed in Section \ref{sec:2} in the absence of the outer density. We give a thorough discussion on solution structure, transition behavior, and long time dynamics, in a specific parameter setting, which allows us to further decouple the model and to derive semi-explicit solutions.

\subsection{Setting and main results}

For simplicity, we consider the one-dimensional radial-symmetric case. The tumor bulk $\Omega(t)=(-R(t),R(t))$ is characterized by its radius $R(t)>0$, and $p$, $c$ and $n$ are even functions in space. {We choose the parameter setting analogous to \cite[Section 3.3]{perthame2014traveling}, which allows us to establish the connection with the traveling wave solution studied there.}

For pressure $p$, the growth rate $G(c)$ is chosen  to be
\begin{equation}\label{specify-G}
    G(c)=\left\{\begin{array}{cc}
g_{+}, & c>\bar{c}, \\
-g_{-}, & c<\bar{c},
\end{array}\right.
\end{equation} where $g_+,\,g_->0$ and $0<\bar{c}<c_B$ are constants. This means cells grow or decay at constant rates, which depends on whether the nutrient concentration is greater than a threshold.

For nutrient $c$, we consider the \textit{in vitro} model \eqref{eq:model-c-in-vitro} and the consumption rate $\psi(n)$ is specified as
\begin{equation}\label{specify-psi}
    \psi(n)=\left\{\begin{array}{ll}
\lambda, & n=1, \\
\lambda n_{c}, & n<1, 
\end{array}\right.
\end{equation} where $\lambda$ and $n_c$ are positive constants with $n_c<1$.

The necrotic core $\Lambda(t)$, defined as the coincidence set for $p$, i.e.  \eqref{def:necrotic-2}, is given by the ansatz
\begin{equation}\label{anstaz-lambda}
    \Lambda(t)=[-r(t),r(t)],\quad \text{for some $r(t)\geq 0$},\qquad \text{or $\quad\Lambda(t)=\emptyset$ is empty}.
\end{equation} This anstaz is made due to the solution symmetry, and it will be verified by later computations.

For density $n$, we have assumed that there is no outer density (i.e. setting $n^{\text{out}}_{\text{init}}=0$ in \eqref{outer-density-initial-data}). Moreover, we expect the density to be unsaturated in the interior of the necrotic core, i.e.  
\begin{equation}\label{claim-density}
    n(x,t)<1,\quad x\in(-r(t),r(t)).
\end{equation}  With \eqref{claim-density}, we can solve for $c$ as the consumption rate $\psi(n)$ is chosen to be constant when $n<1$ \eqref{anstaz-lambda}.

Finally, for $\Omega(t)$, the boundary moving speed is given by \eqref{Darcy-incompresible-revisited} as there is no outer density.

Now we reformulate \eqref{ob-explicit},\eqref{eq:model-c-in-vitro} to the precise equations to solve in the specified setting above. For a fixed time $t$, $\Omega(t)=(-R(t),R(t))$. First, we will assume the necrotic core is non-empty and try to find its radius $r(t)> 0$. 
In the proliferating rim $N(t)=(-R(t),-r(t))\cup(r(t),R(t))$ we have
\begin{equation}\label{eq:proliferation-rim}
    \begin{dcases}
        n=1,\quad &r(t)<|x|<R(t),\\
        -\p_{xx}c+\lambda c=0,\quad &r(t)<|x|<R(t),\\
        -\p_{xx}p=G(c),\quad &r(t)<|x|<R(t),\\
        p\geq0,\quad &r(t)<|x|<R(t),\\
        p=0,\quad &|x|=R(t),\\
        c=c_B,\quad &|x|= R(t).
    \end{dcases}
\end{equation}
In the necrotic core $[-r(t),r(t)]$ and on the interface $|x|=r(t)$, we have
\begin{equation}\label{eq:necrotic-core}
    \begin{dcases}
        n<1,\quad &|x|<r(t),\\
        -\p_{xx}c+\lambda n_c c=0,\quad &|x|<r(t),\\
        p=0,\quad &|x|<r(t),\\
        p=0,\p_xp=0\quad & |x|=r(t),\\
        \text{$c$ and $\p_x c$ are continuous}  \quad &|x|= r(t).
    \end{dcases}
\end{equation}

It turns out that solving \eqref{eq:proliferation-rim}-\eqref{eq:necrotic-core} will lead to a nonlinear equation for $r(t)$, which either has a unique positive solution $r(t)$ or none. The latter case indicates that there is no necrotic core at time $t$. And we can obtain the solution $p$ and $c$ by solving \eqref{eq:proliferation-rim} but on the whole domain $\Omega(t)=\{|x|<R(t)\}$ 
\begin{equation}\label{eq:good}
    \begin{dcases}
                -\p_{xx}c+\lambda c=0,\quad &|x|<R(t),\\
        -\p_{xx}p=G(c),\quad &|x|<R(t),\\
        p=0,\quad &|x|=R(t),\\
        c=c_B,\quad &|x|= R(t),
    \end{dcases}
\end{equation} whose solution $p$ will be non-negative when we can't find a $r(t)\geq0$ for the first case.

\begin{remark}
    Note that while it may be unclear a priori whether we can find a $r>0$ such that \eqref{eq:proliferation-rim}-\eqref{eq:necrotic-core} can be solved, we can actually prove such $r$ is unique, if exists, by some comparison arguments without resorting to the specific form of $G(c)$ in \eqref{specify-G}. This is Proposition \ref{prop:apriori-unique} to be given in Section \ref{sec:3.4}.
\end{remark}

\paragraph{Main results} Investigating the case with fixed $\Omega(t)$, where $R(t)=R$ treated as a constant, leads to our first main result. We characterize the onset and the structure of the necrotic core as follows.
\begin{theorem}\label{thm:3.1}
    There exist constants $0<R_0<R_1<+\infty$, to be defined as roots to explicit equations \eqref{def:R0} and \eqref{def:R1} respectively, such that solutions to equations \eqref{eq:proliferation-rim}-\eqref{eq:good}  with $R(t)=R$ satisfy the following.
    \begin{enumerate}
        \item When $R>R_1$, there exists a necrotic core $\Lambda(t)=[-r,r]$ with $r>0$. Moreover, $r$ is strictly increasing with respect to $R$ with $r(R_1+)=0$. Furthermore, we have a strict inclusion between the necrotic core and the region where $G(c)<0$
        \begin{equation}\label{eq:inclusion}
            \Lambda(t)=\{x\in(-R,R):\, p=0\}\subsetneq\, \{x\in(-R,R):\,G(c(x,t))<0\}.
        \end{equation}
        \item When $R_1>R>R_0$, there is no necrotic core but at the center of the domain (a neighborhood of $x=0$) the nutrient level is less than the critical threshold in \eqref{specify-G} $c<\bar{c}$ and therefore $G(c)<0$. 
        \item When $R_0> R>0$, there is no necrotic core and $c> \bar{c}$ and therefore $G(c)>0$ for all $x\in(-R,R)$.
    \end{enumerate}
    For borderline cases, we have: when $R=R_1$, $p=0$ at $x=0$  but $p>0$ for all $0<|x|<R_1$; when $R=R_0$, $c=\bar{c}$ at $x=0$ but $c>\bar{c}$ for all $0<|x|<R_0$.
\end{theorem}

Intuitively, as the tumor radius $R$ increases, it becomes more difficult for cells at the tumor center to access sufficient nutrients, thus rendering the formation of a necrotic core more likely.

An intriguing phenomenon is that lack of nutrients ($c<\bar{c}$) does not necessarily lead to the formation of a necrotic core. One interpretation might be that the death of cells can be compensated by the neighboring cells, which allows to maintain a positive pressure. This marks a difference compared to some free boundary models for the necrotic core in literature, see the discussion in Section \ref{sec:conclusion}. 

Theorem \ref{thm:3.1} will be proved in Section \ref{sec:3.2}, along which we will derive the expressions for $p$ and $c$.

Next, we can compute the boundary moving speed via \eqref{Darcy-incompresible-revisited}, applying here as 
\begin{equation}\label{Darcy-analytical}
    \frac{d}{dt}R(t)=-\p_xp(R(t),t).
\end{equation} Our second main result (stated below in a formal way) is on the long time behavior.
\begin{theorem}\label{thm:3.2}
In the long time limit, the solution to equations \eqref{eq:proliferation-rim}-\eqref{eq:good} converges to the traveling wave solution constructed in \cite[Section 3.3]{perthame2014traveling}.
\end{theorem}
The exact meaning of the convergence and the proof to Theorem \ref{thm:3.2} will be provided in Section \ref{sec:3.3}. See Figure \ref{fig:dynamics} for an illustration. Before giving the calculations and proofs, we note that the following lemma is useful for this section, whose proof is elementary and is thus omitted. 
\begin{lemma}\label{lem:hyper}
    For $\mu\in(0,1)$, we have
    \begin{enumerate}
        \item $\frac{\cosh(\mu x)}{\cosh (x)}$ is strictly decreasing on $(0,+\infty)$ which equals $1$ at $x=0$ and equals $0$ at  $x=+\infty$.
        \item  $\frac{\sinh(\mu x)}{\sinh (x)}$ is also strictly decreasing on $(0,+\infty)$ which equals $\mu$ at $x=0^+$ and equals $0$ at  $x=+\infty$.
        \item $\frac{\sinh(\mu x)}{\sinh (x)}<\frac{\cosh(\mu x)}{\cosh (x)}$ for all $x>0$.
    \end{enumerate}
\end{lemma}

\begin{figure}[htbp]
    \centering
    \includegraphics[width=0.9\textwidth]{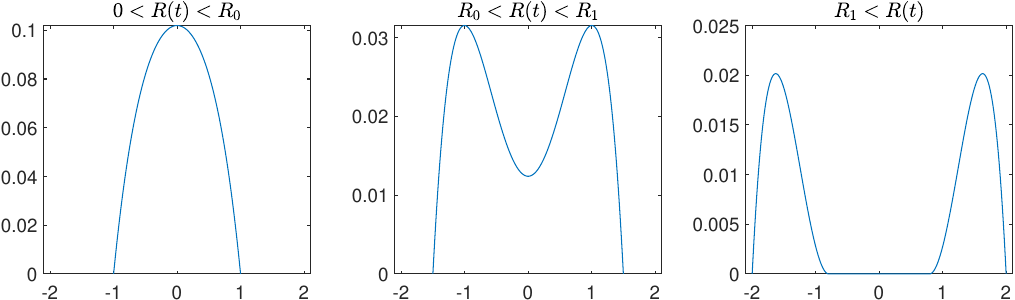}
    \caption{Schematic plot pressure profiles in three phases. From left to right (corresponding to Case 3 -- Case 1), as the tumor radius increases, the nutrient becomes insufficient in the center, and finally the necrotic core forms.}
    \label{fig:pressure}
\end{figure}
\subsection{Onset and structure of the necrotic core}\label{sec:3.2}

In this section, we work out profiles of $p$ and $c$ with a given $\Omega(t)=(-R(t),R(t))$, covering  all the three cases indicated in Theorem \ref{thm:3.1}

\begin{itemize}
    \item Case 1 $R(t)>R_1$. The nutrient is insufficient and the necrotic core forms.
    \item Case 2 $R_1\geq R(t)>R_0$. The nutrient starts to be insufficient while the necrotic core has not formed.
    \item Case 3 $R_0\geq R(t)>0$. The nutrient is sufficient.
\end{itemize}
We will go from Case 1 to Case 3, starting with the most complicated Case 1. See Figure \ref{fig:pressure} for an illustration.

\paragraph{Definition of $R_0,\,R_1$} Before that, we define the two critical radii, whose meanings will be clear in later calculations. $R_0$ is the unique real number such that
\begin{equation}\label{def:R0}
     \frac{1}{\cosh (\sqrt{\lambda}R_0)}=\frac{\bar{c}}{c_B},\qquad R_0>0.
\end{equation} 

$R_1$ is defined to be the unique root of
\begin{equation}\label{def:R1}
    \frac{\cosh (\sqrt{\lambda}(1-\alpha)R_1)}{\cosh (\sqrt{\lambda} R_1)}=\frac{\bar{c}}{c_B},\qquad R_1>0,
\end{equation} where $\alpha$ is a constant given by
\begin{equation}\label{def:alpha}
    \alpha=\sqrt{\frac{g_-}{g_++g_-}}\in(0,1).
\end{equation} 
It is clear that this definition is proper in light of Lemma \ref{lem:hyper}.

We note that $R_0<R_1$ follows from
\begin{equation}
    \frac{1}{\cosh (\sqrt{\lambda}R_0)}=\frac{\cosh (\sqrt{\lambda}(1-\alpha)R_1)}{\cosh (\sqrt{\lambda} R_1)}>\frac{1}{\cosh (\sqrt{\lambda} R_1)},
\end{equation} by definition \eqref{def:R0} and \eqref{def:R1}.

\subsubsection{Case 1: $R>R_1$}\label{sec:321}

For $R(t)>R_1$, we aim to find $r(t)>0$ for which we can solve for $p$ and $c$ from \eqref{eq:proliferation-rim}-\eqref{eq:necrotic-core}. We introduce
\begin{equation}
    \bar{R}(t)=R(t)-r(t),
\end{equation} which is the width of the proliferation rim $N(t)=(-R(t),-r(t))\cup(r(t),R(t))$. For simplicity, we hide the dependence on $t$ hereinafter as $t$ is fixed.
\paragraph{Translation of the domain} By symmetry it suffices to focus on $x\in[0,R]$. We further translate the domain from $(0,R)$ to $(-r,R-r)=(-r,\bar{R})$. The following derivation shares many similarities with that for the traveling wave in \cite[Section 3.3]{perthame2014traveling}, which can be viewed as the case $r=\infty$. The latter point will be elaborated more in Section \ref{sec:3.3}.

    \paragraph{On $(0,\bar{R})$} On the proliferation rim $(0,\bar{R})$, the system is governed by (the translation version of) \eqref{eq:proliferation-rim}.  First we solve the equation for $c$ to get 
\begin{equation}\label{c-0-barR}
c(x)=c_{B} \cosh (\sqrt{\lambda}(x-\bar{R}))+c_{R}^{\prime}\sinh (\sqrt{\lambda}(x-\bar{R})),\quad x\in[0,\bar{R}],
\end{equation}with a parameter $c_R$ to be determined. The solution is strictly increasing and positive provided
\begin{equation}\label{eq:42}
c_{B} \tanh (\sqrt{\lambda} \bar{R}) < c_{R}^{\prime}<\frac{c_{B}}{\tanh (\sqrt{\lambda} \bar{R})}.
\end{equation}

As $G(c)$ is in the piecewise constant form \eqref{specify-G}, we want to determine when $c$ is above or below the critical threshold $\bar{c}$. Suppose $c(x_1)=\bar{c}$, i.e.,
\begin{equation}\label{eq:barc43}c_{B} \cosh \left(\sqrt{\lambda}\left(x_{1}-R\right)\right)+c_{R}^{\prime} \sinh \left(\sqrt{\lambda}\left(x_{1}-R\right)\right)=\bar{c}\end{equation}
Then $c(x)<\bar{c}$ for $x\in(0,x_1)$ and $c(x)>\bar{c}$ for $x\in(x_1,\bar{R})$, as $c$ is strictly increasing on $(0,\bar{R})$ provided that \eqref{eq:42} holds.

Then by $-\p_{xx}p=G(c)$ we have $-\p_{xx}p=-g_-$ on $(0,x_1)$ and $-\p_{xx}p=g_+$ on $(x_1,\bar{R})$. Solve this with the interface condition $p(0)=\p_xp(0)=0$ in \eqref{eq:necrotic-core} (note that $r(t)$ is translated to $0$) and the $C^1$ continuity at $x_1$, and we obtain
\begin{equation}\label{p-0-barR}
p(x)=\begin{dcases}
\frac{g_-}{2}x^2,\ x\in(0,x_1),\\
-\frac{g_+}{2}(x-x_1)^2+g_-x_1(x-x_1)+\frac{x_1^2}{2}g_-,\ x\in(x_1,\bar{R}).
\end{dcases}
\end{equation}Then the boundary condition $p(\bar{R})=0$ leads to an equation for $x_1$
\begin{equation}\label{eq:44px1}
-\frac{g_+}{2}(\bar{R}-x_1)^2+g_-x_1(\bar{R}-x_1)+\frac{x_1^2}{2}g_-=0,
\end{equation}  solving which we obtain
\begin{equation}\label{expression-x1}
    x_1=(1-\sqrt{\frac{g_-}{g_++g_-}})\bar{R}=(1-\alpha)\bar{R},
\end{equation} where $\alpha=\sqrt{\frac{g_-}{g_++g_-}}$ as defined in \eqref{def:alpha}.
At this stage, plugging \eqref{expression-x1} into \eqref{eq:42} leads to one equation but with two unknowns -- $c_R^{\prime}$ and $\bar{R}$. To proceed, we move on to the necrotic core part $(-r,0)$.

\paragraph{On $(-r,0)$.}   First we solve for $c$, using the equation $-\p_{xx}c+\lambda n_c c=0$ and the values at $x=0$ given by \eqref{c-0-barR} thanks to the $C^1$ continuity, which gives
\begin{equation}\label{c-core}
c(x)=c_0\cosh(\xi x)+\frac{c_0^{\prime}}{\xi}\sinh(\xi x), \quad x\in(-r,0),
\end{equation} where $\xi=\sqrt{\lambda n_{c}}$ is a constant, and $c_0$, $c_0^{\prime}$ are values of $c(0),\p_xc(0)$ from \eqref{c-0-barR}, given by
\begin{equation}\label{def-c0c0'}
    \begin{aligned}
        c_0&=c_B\cosh (\sqrt{\lambda} \bar{R})-c_R^{\prime}\sinh(\sqrt{\lambda} \bar{R}),\\
        c_0^{\prime}&=\sqrt{\lambda}(-c_B\sinh (\sqrt{\lambda}\bar{R})+c_R^{\prime}\cosh(\sqrt{\lambda} \bar{R})).
    \end{aligned}
\end{equation}
Note that we have a symmetric condition $\p_xc(-r)=0$, which follows from the radial symmetric setting before the translation, and that $-r$ is the translated origin. Applying $\p_xc(-r)=0$ to \eqref{c-core} we obtain
\begin{equation}\label{eq:46}
c_0-\frac{c_0'}{\xi \tanh (\sqrt{\lambda}r)}=0.
\end{equation}Combining \eqref{def-c0c0'} with \eqref{eq:46}, we get a relationship between $c_B$ and $c_R^{\prime}$, recalling that $\xi=\sqrt{\lambda n_{c}}$
\begin{equation}\label{eq:47}
c_{R}^{\prime}=c_{B} \frac{\beta\cosh (\sqrt{\lambda} \bar{R})+ \sinh (\sqrt{\lambda} \bar{R})}{\beta \sinh (\sqrt{\lambda} \bar{R})+ \cosh (\sqrt{\lambda} \bar{R})}, 
\end{equation} where
\begin{equation}\label{def:beta}
    \beta:=\sqrt{n_c}\tanh (\sqrt{\lambda}r)\in (0,\sqrt{n_c})\subset(0,1),
\end{equation} is an unknown depending on $r$. 

In particular, \eqref{eq:47} ensures the condition \eqref{eq:42}, as
\begin{equation*}
    \tanh (\sqrt{\lambda} \bar{R})<\frac{c_{R}^{\prime}}{c_{B}}=\frac{\beta+ \tanh (\sqrt{\lambda} \bar{R})}{\beta \tanh (\sqrt{\lambda} \bar{R})+ 1}<\frac{1}{\tanh (\sqrt{\lambda} \bar{R})},
\end{equation*} since $\tanh (\sqrt{\lambda}r)\in(0,1)$.

Finally, with all materials gathered, we substitute \eqref{expression-x1} into \eqref{eq:barc43}, and plug in \eqref{eq:47}, which leads to a nonlinear equation for $\bar{R}$
\begin{align}\label{tmp-main-eq}
\mathcal{F}(\bar{R},\beta)=0,
\end{align} where the transcendental function $\mathcal{F}(\bar{R},\beta)$ is defined by 
\begin{equation}\label{def:F}
    \mathcal{F}(\bar{R},\beta):=\frac{\bar{c}}{c_{B}}\Bigl(\cosh (\sqrt{\lambda} \bar{R})+\beta \sinh (\sqrt{\lambda} \bar{R})\Bigr)-\Bigl(\cosh (\sqrt{\lambda}(1-\alpha) \bar{R})+\beta \sinh (\sqrt{\lambda}(1-\alpha) \bar{R})\Bigr).
\end{equation} Note that by \eqref{def:beta}, $\beta$ depends on $\bar{R}$ through $r=R-\bar{R}$.

\paragraph{A nonlinear system} We have reduced the problem of finding the necrotic core to solving a set of nonlinear equations (system) \eqref{def:beta}-\eqref{tmp-main-eq}, summarized as
\begin{align}\label{main-eq}
    &\mathcal{F}(\bar{R},\beta)=0,\\
    &\beta=\sqrt{n_c}\tanh (\sqrt{\lambda}r),\label{recast-def-beta}\\
    &r+\bar{R}=R,\label{recast-relation-r}
\end{align} where $\mathcal{F}(\bar{R},\beta)$ is defined in \eqref{def:F}. Our investigations rely on that the solvabilty for \eqref{main-eq} has been established in \cite{perthame2014traveling} when $\beta$ is fixed.
\begin{lemma}\label{lem:solve-fixed-beta}
    For a fixed $\beta\geq0$, \eqref{main-eq} has a unique root $\bar{R}$ on $(0,+\infty)$, denoted as $f(\beta)$. Moreover, we have 
    \begin{enumerate}
        \item $f(0)=R_1$, where $R_1$ is the critical radius defined as in \eqref{def:R1}.
        \item $f(\beta)$ is strictly decreasing with respect to $\beta$. 
    \end{enumerate}
\end{lemma}
\begin{proof}[Proof of Lemma \ref{lem:solve-fixed-beta}]
    For the uni-solvability, we refer to the last part of the proof of \cite[Theorem 3.1]{perthame2014traveling}.

    Now we write $\mathcal{F}$ as
    \begin{equation}\label{F-rewrite}
         \mathcal{F}(\bar{R},\beta)=\frac{\bar{c}}{c_{B}}\cosh (\sqrt{\lambda} \bar{R})-\cosh (\sqrt{\lambda}(1-\alpha) \bar{R})+\beta\Bigl( \frac{\bar{c}}{c_{B}}\sinh (\sqrt{\lambda} \bar{R})- \sinh (\sqrt{\lambda}(1-\alpha) \bar{R})\Bigr).
    \end{equation}
    To see $f(0)=R_1$, we note that $\mathcal{F}(\bar{R},0)=0$ leads to the same equation as \eqref{def:R1}.

    It is direct to check $f:[0,+\infty)\rightarrow(0,+\infty)$ is an injection. We rewrite $\mathcal{F}(\bar{R},0)=0$ as
    \begin{equation}
        \beta=-\frac{\frac{\bar{c}}{c_{B}}\cosh (\sqrt{\lambda} \bar{R})-\cosh (\sqrt{\lambda}(1-\alpha) \bar{R})}{\frac{\bar{c}}{c_{B}}\sinh (\sqrt{\lambda} \bar{R})- \sinh (\sqrt{\lambda}(1-\alpha) \bar{R})}=: g(\bar{R}).
    \end{equation}By elementary properties of hyperbolic functions in Lemma \ref{lem:hyper}, we can show $g(\bar{R})<0$ for $\bar{R}>R_1$ and $g(\bar{R})>0$ for $\bar{R}\in (R_{\sin},R_1)$, where $R_{\sin}$ is the largest root of the denominator
    \begin{equation*}
        \frac{\bar{c}}{c_{B}}\sinh (\sqrt{\lambda} {R}_{\sin})- \sinh (\sqrt{\lambda}(1-\alpha) R_{\sin})=0.
    \end{equation*} ${R}_{\sin}$ is either zero or a positive number, depending on the relative magnitudes of $\frac{\bar{c}}{c_{B}}$ and $(1-\alpha)$. In either case we can show $R_{\sin}<R_1$ and $g(R_{\sin}+)=+\infty$. Hence $g$ on $(R_{\sin},R_1]$ is the inverse function for $f$ on $[0,+\infty)$. As $g(R_{\sin}+)>g(R_1)$ we know $g$ is strictly decreasing, which implies $f$ is also strictly decreasing. 

\end{proof}

Now we turn to the system \eqref{main-eq}-\eqref{recast-relation-r}.
\begin{proposition}\label{Prop:solve-system}
    Given $R>0$, the nonlinear system \eqref{main-eq}-\eqref{recast-relation-r} has a solution $(r,\bar{R},\beta)$ with $0<r<R$ if and only if $R>R_1$, with  $R_1$ defined in \eqref{def:R1}. Further, when $R>R_1$,  the solution is unique.
\end{proposition}
\begin{proof}[Proof]
    By the map $f$ in Lemma \ref{lem:solve-fixed-beta}, the first equation \eqref{main-eq} is equivalent to $\bar{R}=f(\beta)$. Then we can recast the nonlinear system as a fixed point problem in $r$
\begin{equation*}
    r \quad\rightarrow\quad \beta=\sqrt{n_c}\tanh (\sqrt{\lambda}r) \quad\rightarrow\quad \bar{R}=f(\beta) \quad\rightarrow\quad r=R-\bar{R},
\end{equation*} leading to
\begin{equation*}
    r=R-f(\sqrt{n_c}\tanh (\sqrt{\lambda}r)),
\end{equation*} which is equivalent to
\begin{equation}\label{def:h}
    h(r):=r+f(\sqrt{n_c}\tanh (\sqrt{\lambda}r))=R.
\end{equation}
Then for each $R$ in $h((0,+\infty))$, i.e. the range of $h$ on $(0,+\infty)$, we can find a $r>0$ such that \eqref{eq:proliferation-rim}-\eqref{eq:necrotic-core} admits solution $p$ and $c$, by previous derivation in this section. Therefore, by a priori uniqueness for \eqref{eq:proliferation-rim}-\eqref{eq:necrotic-core} in Proposition \ref{prop:apriori-unique} (in Section \ref{sec:3.4}), we deduce that $h$ is an injection. 

Noting that $h(0)=f(0)=R_1$ by Lemma \ref{lem:solve-fixed-beta} and $h(r)\rightarrow+\infty$ as $r\rightarrow+\infty$, we conclude that $h$ is strictly increasing. And we conclude that \eqref{def:h} has a solution $r>0$ if and only if $R>h(0)=R_1$.

\end{proof}In the last step of the proof, the strictly increasing property of $h$ may also be derived via explicitly working with $f$, which however may lead to complicated expressions involving hyperbolic functions. Hence here we take an indirect approach using the uniqueness in Proposition \ref{prop:apriori-unique}.

\begin{proposition}\label{prop:property-r-barR}
    Denote the solution $r$ obtained in Proposition \ref{Prop:solve-system} as $r(R)$ and $\bar{R}(R)=R-r(R)$ for $R\in (R_1,+\infty)$, then
    \begin{enumerate}
        \item $r(R)$ is strictly increasing with respect to $R$ with $r(R_1+)=0$.
        \item $\bar{R}$ is strictly decreasing with respect to $R$.
        \item  As $R\rightarrow+\infty$, $r$ goes to $+\infty$ and $\bar{R}$ goes to $f(\sqrt{n_c})>0$.
    \end{enumerate}
\end{proposition}

We expect that as the tumor grows the necrotic core becomes larger. But it is worth noting that the proliferating rim becomes \textit{smaller}  in time. Nevertheless, the limit $\bar{R}$ is still positive such that it can be connected to its counterpart in the traveling wave solution, as we will show in Section \ref{sec:3.3}.

\begin{proof}[Proof of Proposition \ref{prop:property-r-barR}]
    The first part follows from $r(R)$ is the inverse of the increasing function of $h(r)$ \eqref{def:h} with $h(0)=R_1$ established in the last part of the proof of Proposition \ref{Prop:solve-system}. And since $h(+\infty)=+\infty$, we know $r(R)$ also goes to infinity as $R\rightarrow+\infty$.
    
    For $\bar{R}$, we note by \eqref{def:h}
    \begin{equation}\label{expression-barR}
        \bar{R}(R)=R-R(r)= f\Big(\sqrt{n_c}\tanh{\big(\sqrt{\lambda}r(R)\big)}\Big),
    \end{equation} which is decreasing as $f$ is decreasing by Lemma \ref{lem:solve-fixed-beta}.

    Finally, taking $R\rightarrow+\infty$ in \eqref{expression-barR} we obtain
    \begin{equation*}
        \lim_{R\rightarrow+\infty}\bar{R}(R)=\lim_{R\rightarrow+\infty}f\Big(\sqrt{n_c}\tanh{\big(\sqrt{\lambda}r(R)\big)}\Big)=f(\sqrt{n_c}).
    \end{equation*}
\end{proof}

\paragraph{Translate back} With $r>0$ determined by Proposition \ref{Prop:solve-system}, now we give the expression of $p$ from \eqref{p-0-barR} translated back to the original domain $(-R,R)$.
\begin{equation}\label{expression-p-transformback}
p(x)=\begin{dcases}
0,\quad &|x|<r,\\
\frac{g_-}{2}(|x|-r)^2,\quad &|x|\in(r,r+x_1),\\
-\frac{g_+}{2}(|x|-r-x_1)^2+g_-x_1(|x|-r-x_1)+\frac{x_1^2}{2}g_-,\quad &|x|\in(r+x_1,{R}),
\end{dcases}
\end{equation} where $x_1=x_1(R)$ is given by \eqref{expression-x1}
\begin{equation}\label{final-expression-x1}
    x_1(R)=(1-\alpha)\bar{R}=(1-\alpha)(R-r),\quad R>R_1.
\end{equation}
For $c$, we can also obtain its expression on $(-R,R)$ translating back from \eqref{c-0-barR} and \eqref{c-core}, which is omitted here for simplicity. We note that  $c(x)<\bar{c}$ for $|x|<r+x_1$ and  the necrotic core $\Lambda=\{|x|<R:p=0\}=[-r,r]$ is strictly included in $\{|x|<R: G(c)<0\}=(-r-x_1,r+x_1)$.
We also observe that for $x\in(r,r+x_1)$, while $G(c)<0$, the pressure is still positive.

\subsubsection{Case 2: $R_1\geq R(t)>R_0$}

When $R<R_1$, a necrotic core is not formed, and therefore the nutrient $c$ solves $-\p_{xx}c+\lambda c=0$ on the whole domain $(-R,R)$, leading to
\begin{equation}\label{expression-c-nonnecrotic}
    c(x)=c_B\frac{\cosh(\sqrt{\lambda}x)}{\cosh(\sqrt{\lambda}R)},\quad x\in(-R,R).
\end{equation} At the center of the tumor, the nutrient level is the lowest, given by
\begin{equation}\label{c(0)-non-necrotic}
    c(0)=c_B\frac{1}{\cosh(\sqrt{\lambda}R)}.
\end{equation}Compare \eqref{c(0)-non-necrotic} with the definition of $R_0$ \eqref{def:R0}, we see $c(0)<\bar{c}$ when $R_1\geq R>R_0$. Since $c(R)=c_B>\bar{c}$ we can find unique $x_1\in(0,R)$ such that $c(x_1)=\bar{c}$, 
or given by
\begin{equation}
    x_1(R):=\frac{1}{\sqrt{\lambda}}\arccosh\left({\frac{\bar{c}}{c_B}\cosh(\sqrt{\lambda}R)}\right),\quad R_0<R\leq R_1. 
\end{equation} We have $c(x)>\bar{c}$ for $|x|\in(x_1,R)$ and $c(x)<\bar{c}$ for $|x|<x_1$. Then $p$ can be solved as
\begin{equation}\label{p-expression-case2}
    p(x)=\begin{dcases}
\frac{g_-}{2}x^2+C(R),\quad &|x|<x_1,\\
-\frac{g_+}{2}(x-x_1)^2+g_-x_1(x-x_1)+\frac{x_1^2}{2}g_-+C(R),\quad &|x|\in(x_1,{R}).
\end{dcases}
\end{equation} with a constant $C(R)$. It is expected that $C(R)>0$ for $R\in(R_0,R_1)$ and $C(R_1)=0$. 

\subsubsection{Case 3: $R_0\geq R(t)>0$}

Finally, when $0<R\leq R_0$, $c(x)\geq c(0)\geq \bar{c}$ for all $x\in(-R,R)$, with the equality only holds at $x=0$ and $R=R_0$. And $p$ is simply given by
\begin{equation}\label{p-expression-case3}
    p=\frac{1}{2}g_+R^2-\frac{1}{2}g_+x^2,\quad x\in(-R,R).
\end{equation}

\subsection{Convergence to traveling wave.}\label{sec:3.3}

With understanding on $p$ and $c$ for a given $R(t)$ at a fixed time,  now we turn to the time evolution.

\paragraph{Boundary moving speed} With the profile of $p$, we can compute the boundary moving speed via \eqref{Darcy-analytical}. For $R(t)>R_1$ we deduce from \eqref{expression-p-transformback}
\begin{align*}
            \frac{d}{dt}R(t)&=-\p_xp(R(t),t)=g_+\bar{R}-(g_++g_-)x_1\\&=g_+\bar{R}-(g_++g_-)(1-\alpha)\bar{R},
\end{align*} where the expression for $x_1$ is given by \eqref{final-expression-x1}. Recall the definition of $\alpha$ \eqref{def:alpha} we conclude
\begin{equation}\label{dRdt}
    \frac{d}{dt}R(t)=(\sqrt{(g_++g_-)g_-}-g_-)\bar{R}(t),\quad R(t)>R_1,
\end{equation} where $\bar{R}(t)$, the size of the proliferation rim, is a function of $R(t)$ described in Proposition \ref{Prop:solve-system}-\ref{prop:property-r-barR}. Similarly we can derive $R'(t)$ for $R(t)<R_1$ using \eqref{p-expression-case2} and \eqref{p-expression-case3}.

\paragraph{Expression of $n$} Clearly, we can apply \eqref{dynamics-n-fixed-x-outer} to derive the expression for density $n$ 
\begin{equation}\label{dynamics-n-analytical}
\begin{dcases}
   n(x,t)=0,\quad t\in(0,t_0),\\
    n(x,t)=1,\quad t\in(t_0,t_1),\\
    n(x,t)=e^{-g_-(t-t_1)},\quad t\in(t_1,+\infty),
\end{dcases}
\end{equation}  where, as in \eqref{def-t0t1}, $t_0=t_0(x)$ and $t_1=t_1(x)$ are defined via
\begin{equation*}
    t_0=\inf\{t\geq 0,\, x<R(t)\},\qquad\qquad\qquad t_1=\inf\{t\geq 0,\, x<r(t)\}.
\end{equation*} Here $n(x,t)=0$ for $t$ in $(0,t_0)$ as we do not consider the outer density. In particular, \eqref{dynamics-n-analytical} implies $n(x,t)<1$ for $x<r(t)$ which verifies the claim \eqref{claim-density}. 

\paragraph{Convergence to traveling wave solution}

When $R(t)>R_1$, by Proposition \ref{prop:property-r-barR} $\bar{R}(t)\geq f(\sqrt{n_c})>0$, which provides a positive lower bound for the boundary moving speed in \eqref{dRdt}. Hence, we deduce that 
\begin{equation}
    R(t)\rightarrow+\infty,\qquad \text{as }t\rightarrow+\infty,
\end{equation} which by Proposition \ref{prop:property-r-barR} again implies 
\begin{equation}
    r(t)\rightarrow+\infty,\quad \bar{R}(t)\rightarrow f(\sqrt{n_c})=:\bar{R}_*>0,\qquad \text{as }t\rightarrow+\infty.
\end{equation} In particular, the limit of $\bar{R}(t)$ implies in \eqref{dRdt} that the boundary moving speed $\frac{d}{dt}R(t)$ tends to be a positive constant.

To see how the profile $p$ and $c$ changes, we use the coordinate in Section \ref{sec:321}, looking at the right part of the solution and translate to $(-r(t),0)$ and $(0,\bar{R}(t))$, which in the limit becomes $(-\infty,0)$ and $(0,\bar{R}_*)$, an infinite large necrotic core and a finite width proliferation rim. 

This is exactly the same setting as the traveling wave solution studied in \cite[Section 3.3]{perthame2014traveling}, and $\bar{R}_*$ coincides with the size of their proliferation rim too. Indeed, the convergence to traveling wave starting with a finite-radius solution is also observed in a numerical way in \cite{perthame2014traveling}.

\begin{figure}[htbp]
    \centering
    \includegraphics[width=0.9\textwidth]{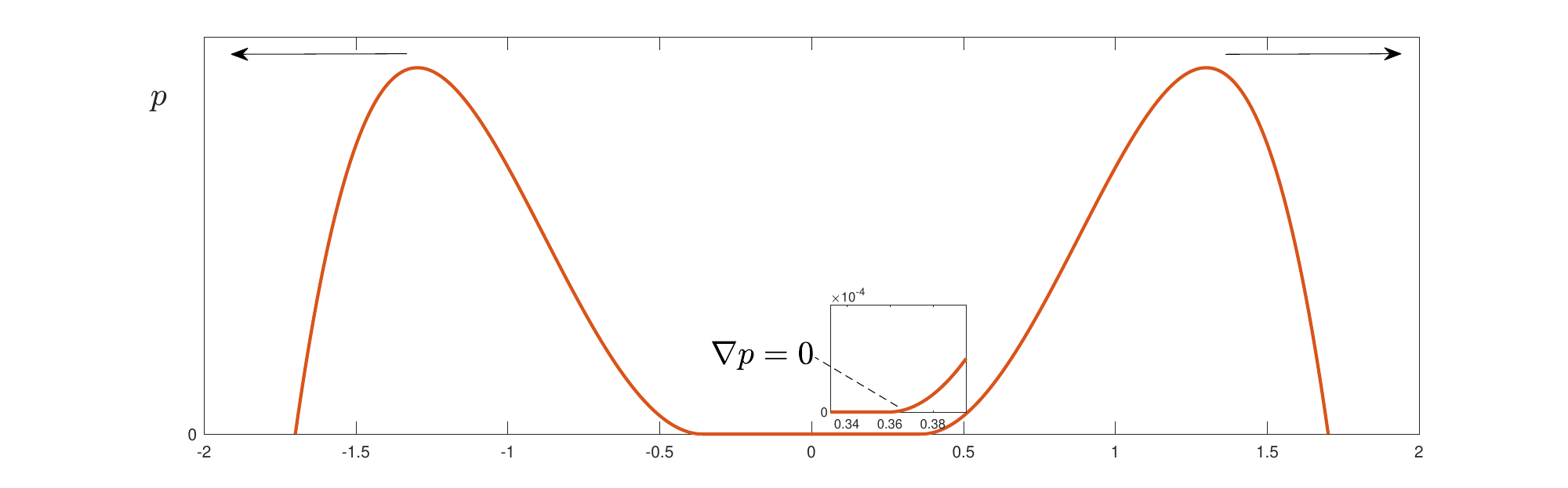}
    \caption{Illustration for the dynamics. Profile of $p$ plotted. A necrotic core forms in the middle, with a regular ($\nabla p=0$) transition to the proliferation rim at the necrotic core boundary. Each of the right part and the left part converges to the profile of the traveling wave solution in \cite[Section 3.3]{perthame2014traveling}.  }
    \label{fig:dynamics}
\end{figure}

\subsection{Uniqueness via comparison arguments}\label{sec:3.4}

Here we prove the uniqueness of system \eqref{eq:proliferation-rim}-\eqref{eq:necrotic-core} via comparison arguments (without knowing the solution exists). More precisely, for a fixed domain $\Omega=(-R,R)$. We reformulate \eqref{eq:proliferation-rim}-\eqref{eq:necrotic-core} as solving
\begin{align}
    \label{eq:tmp-c}
    &\begin{dcases}
        -\p_{xx}c+\lambda c=0,\quad &r<|x|<R,\\
        -\p_{xx}c+\lambda n_c c=0,\quad &|x|<r,\\
        \text{$c$ and $\p_x c$ are continuous}  \quad &|x|= r,\\
        c=c_B,\quad &|x|= R,
    \end{dcases}\\
    &\begin{dcases}
        -\p_{xx}p=G(c),\quad &r<|x|<R,\\
        p\geq0,\quad &r<|x|<R,\\
        p=0,\quad &|x|<r,\\
        p=0,\p_xp=0\quad & |x|=r,\\
        p=0,\quad &|x|=R,\\
    \end{dcases}\label{eq:tmp-p}
\end{align} where $r\in(0,R)$ and functions $p,c$ are all unknowns. 
\begin{proposition}\label{prop:apriori-unique}
    Given $R>0$ and $G(\cdot)$ non-decreasing, solution $(r,p,c)$ to \eqref{eq:tmp-c}-\eqref{eq:tmp-p}  is unique (if exists)
\end{proposition}
The following proof shows that while system \eqref{eq:tmp-c}-\eqref{eq:tmp-p} is nonlinear, there is still some order structure allowing comparison arguments.
\begin{proof}[Proof of Proposition \ref{prop:apriori-unique}]
Suppose there are two different solutions $(r_1,p_1,c_1)$ and $(r_2,p_2,c_2)$. If $r_1=r_2$ then by \eqref{eq:tmp-c} we uniquely solve $c_1=c_2$ and then by \eqref{eq:tmp-p} $p_1=p_2$.  

Hence WLOG we assume $r_1>r_2$, which means the necrotic core of the first solution is larger. We claim then 
\begin{equation}\label{tmp-c1geqc2}
    c_1(x)\geq c_2(x),\quad \forall x\in(-R,R).
\end{equation} Intuitively, this is because cells in the necrotic core consumes less nutrients, then $c_1\geq c_2$ as the first solution as a larger necrotic core.

The claim \eqref{tmp-c1geqc2} can be verified by explicitly solving \eqref{eq:tmp-c}, or by the following comparison argument, which may better illustrate the aforementioned intuition. We write \eqref{eq:tmp-c} as
\begin{equation}
\begin{dcases}
        -\p_{xx}c_i+\psi_i(x) c_i=0,\quad &|x|<R,\\
        c_i=c_B\quad &|x|= R,
    \end{dcases}\qquad\qquad \text{where }\psi_i(x)=\begin{dcases}
        \lambda,\quad &r_i\geq|x|<R,\\
        \lambda{n_c},\quad &|x|<r_i,
    \end{dcases}
    \end{equation} for $i=1,2$. Then as $r_1>r_2$ using $n_c<1$ we have
    \begin{equation}
        \psi_1(x)=\lambda{n_c}+\lambda(1-n_c)\mathbb{I}_{|x|\geq r_1}\leq \lambda{n_c}+\lambda(1-n_c)\mathbb{I}_{|x|\geq r_2}=\psi_2(x).
    \end{equation} Hence $u:=c_1-c_2$ solves
    \begin{equation}
        \begin{dcases}
            -\p_{xx}u+\psi_1(x) u=(\psi_2-\psi_1)c_2\geq0,\quad & |x|<R,\\
            u=0,\quad & |x|=R. 
        \end{dcases}
    \end{equation} Hence as $\psi_1(x)>0$ by maximum principle we have $u\geq0$, which proves claim \eqref{tmp-c1geqc2}.

    As $G(\cdot)$ is non-decreasing, \eqref{tmp-c1geqc2} implies $G(c_1)\geq G(c_2)$. Note that by \eqref{eq:tmp-p} $p_1,p_2$ are solutions to obstacle problems with the same constraint $p\geq 0$ and different sources $G(c_1),G(c_2)$ respectively (when $c$ is radially symmetric and increasing). By the comparison principle for obstacle problems (see e.g. \cite[Chapter 1, Theorem 3.3]{MR1009785}) we have $p_1\geq p_2$, which implies the $r_1\leq r_2$, as the coincidence set of $p_1$ shall be smaller. Intuitively, it means $p_1$ corresponds to a larger growth rate therefore its necrotic core shall be smaller. This is a contradiction with the previous assumption $r_1>r_2$.
\end{proof}

\section{Traveling wave solution with outer density}\label{sec:4}

As discussed in Section \ref{sec:analytical}, the traveling wave solution represents a typical long-term asymptotic behavior of the invasive tumor in the absence of the outer density. Whereas, when the outer density is present, the utility of analytical techniques becomes limited as the models encompass additional layers of complexity. The primary goal of this section is to investigate the existence of traveling wave solutions where the necrotic core and the outer density are connected by a proliferating rim with a finite width. 

In this section, we consider the one-dimensional traveling wave model propagating to the right. We denote the width of the proliferating rim by $R$, and place the boundary point between the necrotic core $\Lambda$ and the proliferating rim $N$ at the origin, hence we have divided the real line into three parts:
\begin{equation}
    \Lambda=(-\infty,0], \quad N=(0,R), \quad \mbox{and} \quad \mathbb R \backslash \Omega = [R, + \infty).
\end{equation}
Thus, $\Omega = (-\infty, R)$ denotes the tumor bulk. The density is saturated in proliferating rim $N$, i.e. $n\equiv 1$, while the density is unsaturated in the interior of the necrotic core $\Lambda$ or the outer region $\mathbb R \backslash \Omega$, i.e. $0\le n <1$. In particular, we expect the density to have a smooth transition from the necrotic core to the proliferating rim but a sharp transition from the proliferating rim to the outer region. See Figure \ref{fig:5} for a schematic plot of the density in the traveling wave model.

In the next, the traveling wave models will be explicitly constructed with both \textit{in vitro} and \textit{in vivo} nutrient dynamics, and we shall prove the existence of traveling wave solutions in each scenario.

\begin{figure}[htbp]
    \centering
    \includegraphics[width=0.5\textwidth]{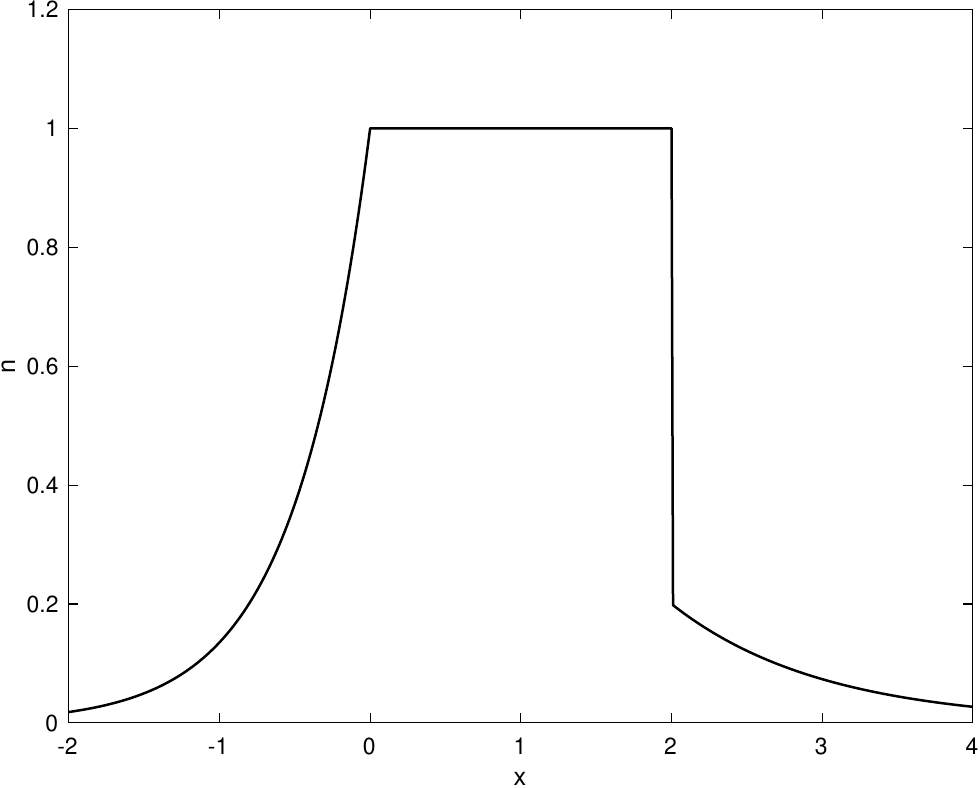}
    \caption{Schematic Drawing for the Model with outer density: In the figure,
    $(-\infty,0]$ is the necrotic core;  $(0,2)$ is the saturated region; $[2,+\infty)$ is the outer region. The cell density is continuous on the boundary between the saturated region and the necrotic core $x=0$;  while there is a jump in the density on the boundary between the saturated region and the outer region $x=2$.}
    \label{fig:5}
\end{figure}

\subsection{Traveling wave model and main results} \label{Sec:4.1}

For simplicity, we start with the one-dimensional tumor growth model, and the generalization to multi-dimensional cases is straightforward as long as we assume radial symmetry. In traveling wave solutions, with a slight abuse of notation, we consider solutions in the following form 
\[
n(x,t)=n(x-\sigma t), \quad p(x,t)=p(x-\sigma t), \quad c(x,t)=c(x-\sigma t),
\]
where $\sigma$ is the traveling wave velocity. 

We first focus on the traveling wave equation for $n$ and $p$. 
By the obstacle problem for the pressure described in Section \ref{Sec:2.1}, we know that $p$ satisfies
\begin{equation} \label{tw:p}
\begin{dcases} 
-\partial_{x x} p=G(c), & \text { in }(0, R), \\ p(0)=p(R)=0, \quad p'(0)=0 . \end{dcases}
\end{equation}
With natural extension, we in fact have $p\equiv 0$ for $x \in (-\infty, 0] \cup [R,+\infty)$.
Next, by density dynamics described in Section \ref{Sec:2.3}, we obtain that 
\begin{equation} \label{tw:n}
\begin{dcases}-\sigma \partial_x n=n G(c), & \text { in } \mathbb{R} \backslash(0, R), \\ n=1, & \text { in }(0, R), \\ n(R+)=n_R, \quad n(0)=1. \end{dcases}
\end{equation}
Note that $n$ has a continuous transition at $x=0$ but jumps from $1$ to $n_R$ at $x=R$. Then, by the velocity law \eqref{Darcy-incompresible-outer}, we have 
\begin{equation} \label{tw:v}
    \sigma \bigl(n(R-)-n(R+)\bigr)=-\partial_xp(R-).
\end{equation}

Next, we move to derive the nutrient models for the traveling wave solutions. For the \textit{in vitro} case, $c$ satisfies
\begin{equation}\label{tw:in_vitro}
    \begin{cases}
    -\partial_{x x} c+\psi(n) c=0, & x<R, \\ c(x)=c_B, &x \ge R.  
    \end{cases}
\end{equation}
And for the \textit{in vivo} case, $c$ satisfies
\begin{equation}\label{tw:in_vivo}
\begin{dcases}
-\partial_{x x} c+\psi(n) c=1_{\{x>R\}}\left(c_B-c\right), & x \in \mathbb{R}, \\ \lim _{x \rightarrow+\infty} c(x)=c_B.
\end{dcases}
\end{equation}

To sum up, we have considered traveling wave solutions of the tumor growth model in two scenarios: 
\begin{itemize}
    \item[$\cdot$] Traveling wave model with the \textit{in vitro} nutrient dynamics composed of \eqref{tw:p}, \eqref{tw:n}, \eqref{tw:v} and \eqref{tw:in_vitro}; 
    \item[$\cdot$] Traveling wave model with the \textit{in vivo} nutrient dynamics composed of \eqref{tw:p}, \eqref{tw:n}, \eqref{tw:v} and \eqref{tw:in_vivo}; 
\end{itemize}

Before proceeding to investigate the existence of traveling wave solutions, we make the following assumptions for the growth rate function $G$ and the consumption function $\psi$:
\begin{assumption}\label{assum_for_psi_G} 
$G$ and $\psi$ satisfy following assumptions:
	\begin{enumerate}
		\item There exists a threshold parameter $\bar{c}>0$ such that
  $$
  \begin{cases}
  G(c)=-g_{-}<0,&c<\bar{c};\\
  G(c)>0,&c\ge\bar{c}.
  \end{cases}
  $$
  Here, $g_-$ is the decay rate when the nutrient is insufficient.  
  $G\in C^1[\bar{c},+\infty)$, $G'\ge 0$ when $x\ge \bar{c}$.
		\item $\psi \in C^1(\mathbb{R})$, $\psi'> 0$, $\psi(0)=0$.
  \item $G$ and $\psi$ satisfy
  $$
  \frac{c_B\|\psi'\|_{L^\infty}\|G'\|_{L^\infty}}{eG(\bar{c})}<1.
  $$
	\end{enumerate}
\end{assumption}

The first item in Assumption \ref{assum_for_psi_G} means that the growth rate $G$ is of an ignition type. In theory, other rate functions can be considered in a similar fashion. The second item is standard and the last one is a technical assumption.

\subsubsection{\textit{In vitro} case}
Consider the traveling wave model with the \textit{in vitro} nutrient dynamics \eqref{tw:p}, \eqref{tw:n}, \eqref{tw:v} and \eqref{tw:in_vitro}. We observe that in this case, the outer density has the following form
\begin{equation}
    n=n_R \exp\left(-\frac{(x-R)G(c_B)}{\sigma}\right), \quad  x \in [R, +\infty).
\end{equation}
With similar arguments as in \cite{perthame2014traveling}, we can conclude that
\begin{theorem}\label{invitrothm}
    For the traveling wave model \eqref{tw:p}, \eqref{tw:n}, \eqref{tw:v} and \eqref{tw:in_vitro},  assume that $c_B>\bar{c}>0,0<n_R<1$ and the first two items of Assumption \ref{assum_for_psi_G} hold, then there exists $\sigma>0$ and $R>0$ such that the system admits a solution with $c$ increasing on $(-\infty,R]$, $n$ increasing on $(-\infty,0]$ and $\lim_{x\rightarrow-\infty}n(x)=0$.  
\end{theorem}
In other words, we observe that here the outer density does not influence the tumor bulk, except amplifying the traveling wave speed by a factor in \eqref{tw:v}. Hence the scenario  is almost the same as the case without outer density, which has been studied in \cite[Theorem 3.1]{perthame2014traveling}. Thus the existence of the traveling wave solution can be shown in an exactly same way. The proof is omitted.

\subsubsection{\textit{In vivo} case}

In comparison, the traveling wave model with the \textit{in vitro} nutrient dynamics \eqref{tw:p}, \eqref{tw:n}, \eqref{tw:v} and \eqref{tw:in_vitro} is more complicated.

We observe that in the outer region, i.e. on the interval $[R,\infty)$, the density equation and the nutrient equation are coupled, and given $n_R$, we are not able to write down the explicit formula for $c$ or $n$. Hence, the outer density has multiple impacts on the traveling wave model by influencing the nutrient distribution and affecting the traveling wave speed.

To investigate the traveling wave solution in this scenario, we need to seek additional qualitative properties. It turns out traveling wave solutions do not exist for all $n_R<1$. Instead, we can find a critical value $\bar n \in (0,1]$ which is determined by the consumption function $\psi$, such that the traveling wave model admits a solution when $n_R$ is less than $\bar n$.

To be  precise, the threshold value $\bar n$ is defined as follows. 

\begin{definition} \label{assump_for_barn}
    Given the consumption function $\psi$, we define $\bar n$ to be largest $ n \in (0, 1]$ such that $\psi(n)$  satisfies
 \begin{equation}\label{cond_for_barn}
          \frac{c_B}{\sqrt{1+\psi({n})}}-\bar{c}\sqrt{1+\psi({n})} > \bar c\sqrt{\psi(1)} 
 \end{equation}
and if we let $C>\bar{c}$ to be constant with
 $$
{\frac{c_B}{\sqrt{1+\psi({n})}}-C\sqrt{1+\psi({n})}}=C \sqrt{\psi(1)},
 $$
 then we have
 $$
 C>\bar{c}e^s, \quad \mbox{where} \quad \tanh(s)=\frac{\psi({n})}{\sqrt{\psi(1)}}.
 $$

\end{definition}
We remark that $\bar n$ is well-defined as long as $0<\bar{c}<{c_B}/(1+\sqrt{\psi(1)})$ in which case as long as $n$ is sufficiently close to $0$, the two conditions in Definition \ref{assump_for_barn} are clearly satisfied. In fact, it is possible that for some consumption function, the associated threshold $\bar n =1$, while for some other consumption function, we have $\bar n <1$.

Now we are ready to state the main theorem
\begin{theorem}\label{invivothm}
Consider the traveling wave model  \eqref{tw:p}, \eqref{tw:n}, \eqref{tw:v} and \eqref{tw:in_vivo}. Assume that $c_B>0$ and $0<\bar{c}<{c_B}/(1+\sqrt{\psi(1)})$. In addition, assume that Assumption \ref{assum_for_psi_G}  holds, and $\bar n$ is the threshold value as in Definition \ref{assump_for_barn}. Then when $0\le n_R<\bar{n}$, there exist $\sigma>0$ and $R>0$ such that the traveling wave model admits a solution with $c$ increasing on $\mathbb{R}$, $n$ increasing on $(-\infty,0]$ and $\lim_{x \rightarrow-\infty}n(x)=0$.
	\end{theorem}

 This theorem indicates that the outer density needs to satisfy a smallness condition for traveling wave solutions to exist.  The schematic drawing for such a typical solution is shown in Figure \ref{fig:6}:
      \begin{figure}[htbp]
    \centering
    \includegraphics[width=0.5\textwidth]{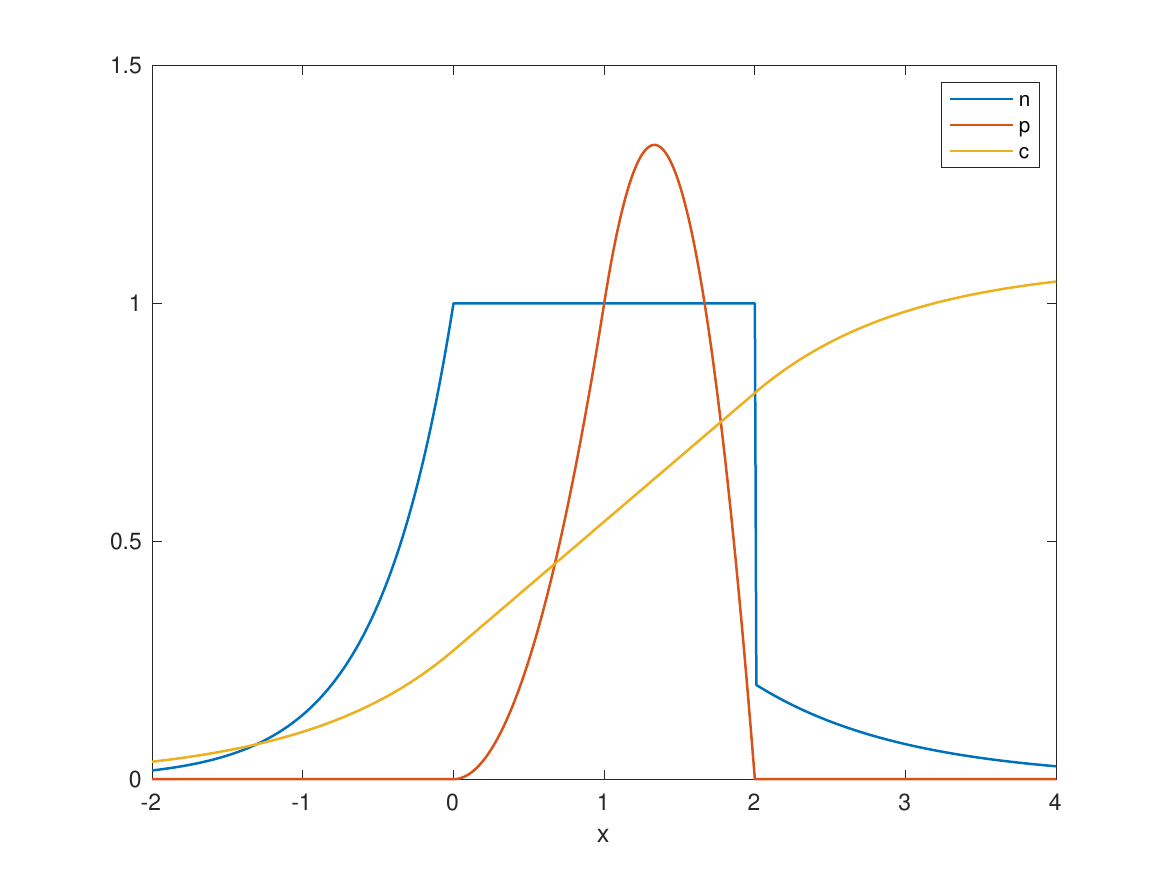}
    \caption{Schematic Drawing for the solution: In the figure,
    $(-\infty,0]$ is the necrotic core;  $(0,2)$ is the saturated region; $[2,+\infty)$ is the outer region. The cell density $n$ is similar to the one in Figure \ref{fig:5}. The pressure $p$ is supported in $(0,2)$ and satisfies $p'(0)=0$. The nutrient density is increasing and satisfies $c(-\infty)=0,\,c(+\infty)=c_B$.}
    \label{fig:6}
\end{figure}

The proof of Theorem \ref{invivothm} is partially inspired by  Theorem 3.1 in \cite{perthame2014traveling}, whereas substantial work is needed to treat the outer density. To this end, we must examine the behavior of $c$ and $n$ on the outer region $[R,+\infty)$. In particular, in the process of constructing the traveling wave solutions, we need to consider the system
\begin{equation}\label{problem_for_right}
	\begin{dcases}
	-\sigma \partial_xn=nG(c),&x>R;\\
    -\partial_{xx}c+\psi(n)c=(c_B-c),&x>R;\\
	n(R+)=n_R,\,c(R)=c_R,\,\partial_xc(R)=c'_R.
	\end{dcases}
	\end{equation}
More precisely, when given the traveling wave speed $\sigma$ and $c(R) = c_R$, we must determine admissible values for $c'_R$  so that $c(x)$ is increasing on $[R,+\infty)$ and approaches the boundary value $c(+\infty)=c_B$. It turns out the admissible value for $c'_R$ is unique, which is denoted by $B(\sigma,c_R)$ in the next, and we have the following lemma.
\begin{lemma}\label{lemma_for_right}
Consider the system (\ref{problem_for_right}). Assume $\sigma>0,\,R>0,\,n_R\in [0,1)$, Assumption \ref{assum_for_psi_G} holds, and $\bar{c}<c_R<c_B/(1+\psi(n_R))$. 
	Then, there exists a continuous mapping $(\sigma,\,c_R)\mapsto B(\sigma,c_R)\in \mathbb{R}$ such that: 
 \begin{enumerate}
 \item In the solution of the system (\ref{problem_for_right}), $c$ is non-decreasing with $c(+\infty)=c_B$ if and only if $c_R'=B(\sigma,c_R)$;
 \item $B(\sigma,c_R)\in\bigl[\frac{c_B}{\sqrt{1+\psi(n_R)}}-\sqrt{1+\psi(n_R)}c_R,\,c_B-c_R\bigr]$;
 \item $B(\sigma,c_R)$ is decreasing with respect to $c_R$.
 \end{enumerate}
\end{lemma}
The proofs of Theorem \ref{invivothm} and Lemma \ref{lemma_for_right} are given in the next subsection. 
\subsection{Main proofs}
\subsubsection{Proof of Theorem \ref{invivothm} }
\begin{proof}[Proof of Theorem \ref{invivothm}]
The framework of the proof can be represented by the following schematic drawing:
  \begin{figure}[htbp]
    \centering
    \includegraphics[width=1.0\textwidth]{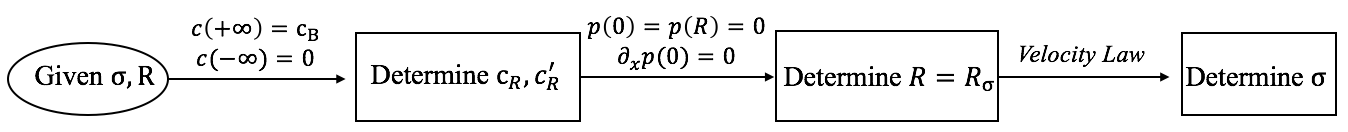}
    \caption{Framework of the Proof}
    \label{fig:7}
\end{figure}

There are four free variables $\sigma,\,R,\,c(R)=c_R,\,\partial_xc(R)=c'_R$ in this system. The strategy of the proof is to choose these variables properly to guarantee the existence of a solution. First, we fix $\sigma$ and $R$. By the equation for $c$ (\ref{tw:in_vivo}) combined with the equation for $n$ $(\ref{tw:n})$, we can determine the values $c_R$ and $c'_R$ to make $c$ satisfy the asymptotic conditions $c(-\infty)=0$ and $c(+\infty)=c_B$. Second, by the equation for $p$ (\ref{tw:p}), the extra boundary condition would provide an extra equation for $R$, which will determine $R=R_\sigma$ corresponding to the fixed $\sigma$. Last, by utilizing the velocity law (\ref{tw:v}), we can finally determine the traveling wave speed $\sigma$. And the theorem follows by studying the fixed point problem.

We divide the proof into two main steps. In the first step, we construct the solution as shown in Fig \ref{fig:6}, i.e. we build the solution successively in each region and determine $c_R,c_R'$ and $R$. In the second step, we utilize the velocity law to construct a fixed point argument to determine the variable $\sigma$. 
	\paragraph{\textbf{Step 1. Construct the solution successively on each region.}}
 In this step, we aim to construct the solution as shown in Fig \ref{fig:6}. Suppose $R>0,c_R\in(\bar{c},\frac{c_B}{1+\psi(n_R)}),c'R>0$ are undetermined. We will first construct the solution using these undetermined coefficients. Then we use the conditions $c(+\infty)=c_B$ and $c(-\infty)=0$ to determine the maps $R\mapsto c_R, R\mapsto c'_R$. Last we use the equations satisfied by $p$ to determine $R$. 

	  On the outer region $[R,+\infty)$: By Lemma \ref{lemma_for_right}, when $c_R'=B(\sigma,c_R)$,  $(n(x),c(x))$ meets all the requirements. In addition, $p(x)=0$.
	
	  On the saturated region $(0,R)$: $n(x)=1$. Setting $\tilde{c}'_R=c'_R/\sqrt{\psi(1)}$, we obtain
	\begin{equation}\label{c_explixit}
		c(x)=c_R \cosh (\sqrt{\psi(1)}(x-R))+\tilde{c}_R^{\prime} \sinh (\sqrt{\psi(1)}(x-R)),
	\end{equation}
	and
	$$
p(x)=-\int_0^x \int_0^y G(c(z)) d z d y=-\int_0^x(x-z) G(c(z)) d z .
	$$
 by  equation (\ref{tw:in_vivo}) and equation (\ref{tw:p}) respectively. We note here that we do not use the boundary condition $p(R)=0$ to obtain the explicit formula for $p$ here. The extra relation $p(R)=0$ implies an equation for $R$: 
	\begin{equation}\label{cond_for_R}
 \begin{aligned}
 \int_0^R(R-z) G(c(z)) d z&=
 \int_0^R(R-z) G(c_R \cosh (\sqrt{\psi(1)}(z-R))+\tilde{c}_R^{\prime} \sinh (\sqrt{\psi(1)}(z-R))) d z\\&=
		\int_0^1 s G\left(c_R \cosh (\sqrt{\psi(1)} R s)-\tilde{c}_R^{\prime} \sinh (\sqrt{\psi(1)} R s)\right) d s=0.
 \end{aligned}
\end{equation}
We note here that this equation implies $c(0)<\bar{c}$ naturally.

	 On the necrotic core $(-\infty,0]$: $p(x)=0,n(x)=e^{\frac{g_-}{\sigma}x}$. According to Lemma \ref{lemma_for_left} in Appendix A, the requirements of $c$ (including $c$ is increasing and $c(-\infty)=0$) implies 
  \begin{equation}\label{c0_relation}
  \partial_xc(0)=A(\sigma)c(0), 
  \end{equation}
  where $0<A(\sigma)<\sqrt{\psi(1)}$ is continuous with respect to $\sigma$ and can be extended by continuity to $A(0)=0$.
  By (\ref{c_explixit}) we know that
  $$
  c(0)=c_R\cosh(\sqrt{\psi(1)}R)-\tilde{c}_R'\sinh(\sqrt{\psi(1)}R)
  $$
  and
  $$
  \partial_xc(0)=-\sqrt{\psi(1)}c_R\sinh(\sqrt{\psi(1)}R)+\sqrt{\psi(1)}\tilde{c}_R'\cosh(\sqrt{\psi(1)}R).
  $$
  Therefore, (\ref{c0_relation}) is equivalent to
	\begin{equation}\label{cond_for_cRpi}
		\tilde{c}_R^{\prime}=c_R \frac{A(\sigma) \cosh (\sqrt{\psi(1)} R)+\sqrt{\psi(1)} \sinh (\sqrt{\psi(1)} R)}{A(\sigma) \sinh (\sqrt{\psi(1)} R)+\sqrt{\psi(1)} \cosh (\sqrt{\psi(1)} R)}.
	\end{equation}
	 The relations satisfied by $c_R$ and $c'_R$ can be written as the following equation set:
 \begin{equation}\label{equ_cRcR'}
 \begin{cases}
     \tilde{c}_R^{\prime}&=c_R \frac{A(\sigma) \cosh (\sqrt{\psi(1)} R)+\sqrt{\psi(1)} \sinh (\sqrt{\psi(1)} R)}{A(\sigma) \sinh (\sqrt{\psi(1)} R)+\sqrt{\psi(1)} \cosh (\sqrt{\psi(1)} R)},\\
     \tilde{c}_R^{\prime}&=\frac{B(\sigma,c_R)}{\sqrt{\psi(1)}}.
 \end{cases}
 \end{equation}
 We should show that system (\ref{equ_cRcR'}) admits a unique solution. 
 
 Since $0<A(\sigma)<\sqrt{\psi(1)}$, we have
 $$
 \tanh(\sqrt{\psi(1)}R)<\frac{A(\sigma) \cosh (\sqrt{\psi(1)} R)+\sqrt{\psi(1)} \sinh (\sqrt{\psi(1)} R)}{A(\sigma) \sinh (\sqrt{\psi(1)} R)+\sqrt{\psi(1)} \cosh (\sqrt{\psi(1)} R)}<1.
 $$
 Since $B(\sigma,c_R)/(\sqrt{\psi(1)}c_R)$ is decreasing with respect to $c_R$, it suffices to show that the range of 
 $B(\sigma,c_R)/(\sqrt{\psi(1)}c_R)$ covers $[\tanh(\sqrt{\psi(1)}R),1]$ to guarantee existence of solutions to system (\ref{equ_cRcR'}). By Lemma \ref{lemma_for_right}, we know the range of 
 $B(\sigma,c_R)/(\sqrt{\psi(1)}c_R)$ covers
 $$
	\begin{aligned}
	  &[\frac{c_B-c_B/(1+\psi(n_R))}{\sqrt{\psi(1)}c_B/(1+\psi(n_R))},\frac{c_B/\sqrt{1+\psi(n_R)}-\sqrt{(1+\psi(n_R)}\bar{c}}{\sqrt{\psi(1)}\bar{c}}] \\
      =
	&[\frac{\psi(n_R)}{\sqrt{\psi(1)}},\frac{1}{\sqrt{\psi(1)}}\frac{\frac{c_B}{\sqrt{1+\psi(n_R)}}-\bar{c}\sqrt{1+\psi(n_R)}}{\bar{c}}]. \qquad
	\end{aligned}
 $$
	
 Define $R_0=R_0(n_R)>0$ such that 
 
 $$\tanh(\sqrt{\psi(1)}R_0)=\frac{\psi(n_R)}{\sqrt{\psi(1)}}.$$
 By Definition \ref{assump_for_barn}, we know that when $0<n_R<\bar{n}$ and $R>R_0(n_R)$, $$ [\tanh(\sqrt{\psi(1)}R),1] \subset[\frac{\psi(n_R)}{\sqrt{\psi(1)}},\frac{1}{\sqrt{\psi(1)}}\frac{\frac{c_B}{\sqrt{1+\psi(n_R)}}-\bar{c}\sqrt{1+\psi(n_R)}}{\bar{c}}].
 $$
 holds. Thus $B(\sigma,c_R)/(\sqrt{\psi(1)}c_R)$ covers $[\tanh(\sqrt{\psi(1)}R),1]$. Therefore we define a mapping $R\mapsto c_R(R>R_0)$ such that $c_R$ is the solution to (\ref{equ_cRcR'}). By the relation (\ref{cond_for_cRpi}) we know that $\tilde{c}'_R/c_R$ is continuous and increasing with respect to $R$. Hence $B(\sigma,c_R)/(\sqrt{\psi(1)}c_R)$is increasing with respect to $R$ and that the map $R\mapsto c_R$ is decreasing. It is not hard to verify that the map is also continuous.
	
	Now it remains to prove that there exists $R_\sigma>R_0$ such that the relation (\ref{cond_for_R}) holds. We define
	\begin{equation}\label{gamma_def}
	\begin{aligned}
	\gamma_\sigma(R, s)  &:=c_R \cosh (\sqrt{\psi(1)} R s)-\tilde{c}_R^{\prime} \sinh (\sqrt{\psi(1)} R s) \\
	& =c_R \frac{A(\sigma) \sinh (\sqrt{\psi(1)} R(1-s))+\sqrt{\psi(1)} \cosh (\sqrt{\psi(1)} R(1-s))}{A(\sigma) \sinh (\sqrt{\psi(1)} R)+\sqrt{\psi(1)} \cosh (\sqrt{\psi(1)} R)}\\
	& :=c_R\tilde{\gamma}_\sigma(R,s).
	\end{aligned}
	\end{equation}
	Then, the relation (\ref{cond_for_R}) can be written as
	$$
	\int_0^1 s G\left(\gamma_\sigma\left(R_\sigma, s\right)\right) d s=0.
	$$
	Since $c_R$ and $\tilde{\gamma}_\sigma(R,s)$ are non-negative and decreasing with respect to $R$, we can conclude that $\gamma_\sigma(R,s)$ is decreasing with respect to $R$.
	
	From (\ref{gamma_def}) we also know that $\gamma_\sigma(R,s)$ is decreasing with respect to $s$. Then when $R=R_0$, $$\gamma_\sigma(R_0,s)\ge \gamma_\sigma(R_0,1)=c_{R_0}\frac{\sqrt{\psi(1)}}{A(\sigma) \sinh (\sqrt{\psi(1)} R_0)+\sqrt{\psi(1)} \cosh (\sqrt{\psi(1)} R_0)}\ge c_{R{0}}/e^{\sqrt{\psi(1)}R_0}.$$
   By the decreasing property of the map $R\mapsto c_R$, it is easy to find that $c_{R_0}>C$ in Assumption \ref{assump_for_barn}. Hence $\gamma_\sigma(R_0,s)>\bar{c}$. Therefore $\int_0^1 s G\left(\gamma_\sigma\left(R_0, s\right)\right) d s>0.$
	
	On the other hand, when $R\rightarrow+\infty$, we have $\gamma_\sigma(R,s)\sim e^{-\sqrt{\psi(1)}Rs}\rightarrow0(\text{for } s>0)$. Then $
	\int_0^1 s G\left(\gamma_\sigma(R, s)\right) d s \rightarrow G(0) / 2<0.
	$ Thus by continuity and monotonicity, we conclude on the existence of a unique $R_\sigma>0$ such that (\ref{cond_for_R}) is satisfied. In addition, by continuity of $\sigma\mapsto A(\sigma)$ and $\sigma\mapsto B(\sigma,c_R) $, we deduce that $\sigma\mapsto R_\sigma$ is continuous.

	\paragraph{\textbf{Step 2. Determine the traveling wave speed $\sigma$ through the velocity law.}}
	We can rewrite the fixed point problem for $\sigma$ as
	\begin{equation}\label{fix_point_newform}
		(1-n_R)\sigma=R_\sigma \int_0^1 G\left(\gamma_\sigma\left(R_\sigma, s\right)\right) d s.
	\end{equation}
	First,  since $A(\sigma)\rightarrow 0$ when $\sigma\rightarrow0$, we know that $R_\sigma\rightarrow\bar{R}>R_0$ which is a solution to
	$$
	\int_0^1 s G\left(\gamma_\sigma\left(\bar{R}, s\right)\right) d s=\int_0^1 s G\left(c_{\bar{R}} \frac{\cosh \left(\sqrt{\psi(1)} \bar{R}(1-s)\right)}{\cosh \left(\sqrt{\psi(1)} \bar{R}\right)}\right) d s=0
	$$
	
	Since the map $s\mapsto \gamma_\sigma(\bar{R},s)$ is decreasing, there exists $s_0\in(0,1)$ such that $G(\gamma_\sigma(\bar{R},s))\ge 0$ for all $s<s_0$ and $G(\gamma_\sigma(\bar{R},s))< 0$ for all $s>s_0$. Thus 
	$$
	0=\int_0^1 s G\left(\gamma_\sigma\left(\bar{R}, s\right)\right) d s<s_0 \int_0^1 G\left(\gamma_\sigma\left(\bar{R}, s\right)\right) d s
	$$
	Then when $\sigma\rightarrow0$, the right-hand side of the velocity law (\ref{fix_point_newform}) is strictly positive and hence larger than the left-hand side of it.
	
	On the other hand, from (\ref{gamma_def}) we have $$
	\gamma_\sigma(R, s) \leq \frac{c_R}{\cosh (\sqrt{\psi(1)} R s)}\leq  \frac{c_B}{\cosh (\sqrt{\psi(1)} R s)}
	.$$ It is directly to verify that,  there exists a unique $R_b>0$ such that 
	$$
	\int_0^1 s G\left(\frac{c_B}{\cosh \left(\sqrt{\psi(1)} R_b s\right)}\right) d s=0.
	$$
	by the monotonicity of $G$. Therefore, for all $R\ge R_b$, by the monotonicity of $G$, we have
	$$
	\begin{aligned}
	\int_0^1 s G\left(\gamma_\sigma(R, s)\right) d s & \leq \int_0^1 s G\left(\frac{c_B}{\cosh (\sqrt{\psi(1)} R s)}\right) d s \\
	& \leq \int_0^1 s G\left(\frac{c_B}{\cosh \left(\sqrt{\psi(1)} R_b s\right)}\right) d s=0 .
	\end{aligned}
	$$
	Thus $
	\int_0^1 s G\left(\gamma_\sigma\left(R_\sigma, s\right)\right) d s=0
	$ implies that $0<R_\sigma\le R_b$. It follows that the right-hand side of (\ref{fix_point_newform}) is bounded. When $\sigma\rightarrow+\infty$, the left-hand side of the velocity law (\ref{fix_point_newform}) is larger than the right-hand side of it. 
 
 By continuity, we know that the fixed point problem (\ref{fix_point_newform}) admits a solution $\sigma>0$.
\end{proof}
\subsubsection{Proof of Lemma \ref{lemma_for_right} }
\begin{proof}[Proof of Lemma \ref{lemma_for_right}]
 
 It is direct to show that given $c(x)$, $n(x)$ can be represented as
	$$
	n(x)=n_Re^{\int_{R}^{x}-\frac{G(c(y))}{\sigma}dy},\quad x\ge R,
	$$
	according to $-\sigma \partial_xn=nG(c)$ for $x>R$ and $n(R+)=n_R$. Thus we can reformulate (\ref{problem_for_right}) as the following ODE system:
	\begin{align}
	&\partial_xc=u,&c(R)=c_R;\label{c_ODE_for_right}\\
	&\partial_xu=(c-c_B)+\psi(n_Re^{\int_{R}^{x}-\frac{G(c(y))}{\sigma}dy})c,&u(R)=c_R'.\label{u_ODE_for_right}
	\end{align}
	
     Given $\bar{c}<c_R<c_B/(1+\psi(n_R))$. For $c_R'\ge 0$, there are two kinds of solutions to the reformulated system \eqref{c_ODE_for_right}-\eqref{u_ODE_for_right}:
	\begin{enumerate}
		\item[I.]  there exists $z\in[R,+\infty)$ such that $u(z)=0$ and $u(x)>0$ for all $x\in[R,z)$;
		
		\item[II.]  $u(x)>0$ for all $x\in[R,+\infty)$.
	\end{enumerate}

Our goal is to determine $c_R'$ such that the solution to (\ref{c_ODE_for_right})-(\ref{u_ODE_for_right}) satisfies $u(x)>0$ when $x\ge R$ and $u(+\infty)=0$, which is a "critical" solution of these two kinds of solutions. We define
$$
\ell:=\sup\{c'_R\ge 0: \text{The solution to (\ref{c_ODE_for_right})-(\ref{u_ODE_for_right}) is of Type I}\}.
$$
The remainder of the proof is divided into 4 steps. In the first and the second step, we show that $\ell\le c_B-c_R<+\infty$ and $\ell\ge \frac{c_B}{\sqrt{1+\psi(n_R)}}-\sqrt{1+\psi(n_R)}c_R$ respectively. In the third step, we show that if we take $c_R'=\ell$, the corresponding solution meets all the requirements above.  In the fourth step, we show that there does not exist $\tilde{\ell}\not=\ell$ such that the solution with $u(R)=\tilde{\ell}$ can meet the requirements. Hence we can define $B(\sigma, c_R):=\ell$. Additionally, we show that  $B(\sigma, c_R)$ is decreasing with respect to $c_R$ in the last step.

\paragraph{\textbf{Step 1. Derive a lower bound of $\ell$.}}
We aim to prove
$$
\ell\ge \frac{c_B}{\sqrt{1+\psi(n_R)}}-\sqrt{1+\psi(n_R)}c_R
$$ 
in this step. It suffices to show that when $c_R'<\frac{c_B}{\sqrt{1+\psi(n_R)}}-\sqrt{1+\psi(n_R)}c_R$, the solution $(c,u)$ corresponding to $c_R'$ is of type I. 

We consider the following ODE system:
\begin{align}
&\partial_x\tilde{c}=\tilde{u},&\tilde{c}(R)=c_R;\label{ctil_ODE_for_right}\\
&\partial_x\tilde{u}=(\tilde{c}-c_B)+\psi(n_R)\tilde{c},&\tilde{u}(R)=c_R'.\label{util_ODE_for_right}
\end{align}
We notice that when $$c_R'= \frac{c_B}{\sqrt{1+\psi(n_R)}}-\sqrt{1+\psi(n_R)}c_R,$$ system (\ref{ctil_ODE_for_right})-(\ref{util_ODE_for_right}) admits a solution 
$$
\tilde{c}(x)=\frac{c_B}{1+\psi(n_R)}-\frac{c_R'}{\sqrt{1+\psi(n_R)}}e^{-\sqrt{1+\psi(n_R)}(x-R)},
$$
which satisfies $\partial_x\tilde{u}(x)>0\,(x\ge R)$ and $\tilde{u}(+\infty)=0$. 

If the solution $(c,u)$ is of type II, then $c$ is non-decreasing and 
$$
\psi(n_Re^{\int_{R}^{x}-\frac{G(c(y))}{\sigma}dy})<\psi(n_R).
$$
Thus when $$c_R'< \frac{c_B}{\sqrt{1+\psi(n_R)}}-\sqrt{1+\psi(n_R)}c_R,$$ the solution $(c,u)$ of system(\ref{c_ODE_for_right})-(\ref{u_ODE_for_right})  should be bounded from above by the solution $(\tilde{c},\tilde{u})$ of system  (\ref{ctil_ODE_for_right})-(\ref{util_ODE_for_right}), i.e. $u\le \tilde{u}$ and $c\le \tilde{c}$. It implies that $$\partial_xu\le \partial_x\tilde{u} \text{ and } u-\tilde{u}\le u(R)-\tilde{u}(R)=c_R'- \frac{c_B}{\sqrt{1+\psi(n_R)}}-\sqrt{1+\psi(n_R)}c_R<0.$$ By the property that $\tilde{u}(+\infty)=0$, we know that $u(+\infty)<0$ and the solution $(c,u)$ is of Type I, which contradicts with the assumption that $(c,u)$ is of Type II. Therefore, when $c_R'<\ell$, the solution $(c,u)$ corresponding to $c_R'$ is of Type I.

Thus we obtain a lower bound of $\ell$ that $$\ell\ge\frac{c_B}{\sqrt{1+\psi(n_R)}}-\sqrt{1+\psi(n_R)}c_R.$$

\paragraph{ \textbf{Step 2. Derive an upper bound of $\ell$}.}
We aim to prove
$$
\ell\le c_B-c_R
$$ 
in this step. It suffices to show that when $c_R'>c_B-c_R$, the solution $(c,u)$ corresponding to $c_R'$ is of type II. 

We consider the following ODE system:
\begin{align}
\partial_x\bar{c}&=\bar{u},&\bar{c}(R)=c_R;\label{ctil_ODE_for_right_upp}\\
\partial_x\bar{u}&=(\bar{c}-c_B),&\bar{u}(R)=c_R'.\label{util_ODE_for_right_upp}
\end{align}
 
When $c_R'= c_B-c_R$, system (\ref{ctil_ODE_for_right_upp})-(\ref{util_ODE_for_right_upp}) admits a solution 
$$
\bar{c}(x)=c_B-c_R'e^{-(x-R)},
$$
which satisfies $\partial_x\tilde{u}(x)>0$ when $x\ge R$ and $\tilde{u}(+\infty)=0$. Then following a similar analysis to the previous step, we can obtain an upper bound of $\ell$ that $$\ell\le c_B-c_R.$$

\paragraph{\textbf{Step 3. Examine the case  $c_R'=\ell$.}}
We first show that the solution $(c_{\ell},u_{\ell})$ corresponding to  $c_R'=\ell$ is of type II. If not, there exists $z\in[R,+\infty)$ such that $u_{\ell}(z)=0$ and $u_{\ell}(x)>0$ for all $x\in[R,z)$. Then $\partial_xu_{\ell}(z)\le 0$. By directly computations, we have 
$$
\begin{aligned}
\partial_{xx}u_{\ell}(z)&=u_{\ell}(z)+\psi(n_Re^{\int_{R}^{z}-\frac{G(c_{\ell}(y))}{\sigma}dy})u_{\ell}(z)+\psi'(n_Re^{\int_{R}^{z}-\frac{G(c_{\ell}(y))}{\sigma}dy})(-n_R\frac{G(c_{\ell}(z))}{\sigma})c_{\ell}(z)\\
&=\psi'(n_Re^{\int_{R}^{z}-\frac{G(c_{\ell}(y))}{\sigma}dy})(-n_R\frac{G(c_{\ell}(z))}{\sigma})c_{\ell}(z).
\end{aligned}
$$
Since $c_{\ell}(z)>c_R>\bar{c}$, we have $G(c_{\ell}(z))>0$. Therefore, $\partial_{xx}u_{\ell}(z)<0$. The facts $u_{\ell}(z)=0$, $\partial_xu_{\ell}(z)\le 0$ and $\partial_{xx}u_{\ell}(z)<0$ yield that there exists $\Delta z>0$ such that $u_{\ell}(x)<0$ in $(z,z+\Delta z)$ and hence $c_{\ell}(z+\Delta z)<c_{\ell}(z)$. We denote $c_{\ell}(z)-c_{\ell}(z+\Delta z)=\Delta c>0$.  

By the continuous dependence of the solution to (\ref{problem_for_right}) on the initial data $(c_R,c_R',n_R)$, we know that there exists $\Delta \ell>0$, such that when $|c_R'-\ell|<\Delta \ell$, the corresponding solution $u$ satisfies $|c(x)-c_{\ell}(x)|<\frac{\Delta c}{2}$ in the interval $[R,z+\Delta z]$. It implies that 
$$
\begin{aligned}
c(z+\Delta z)-c(z)&=(c(z+\Delta z)-c_{\ell}(z+\Delta z))+(c_{\ell}(z)-c(z))+(c_{\ell}(z+\Delta z)-c_{\ell}(z))\\
&\le|c(z+\Delta z)-c_{\ell}(z+\Delta z)|+|c_{\ell}(z)-c(z)|+(c_{\ell}(z+\Delta z)-c_{\ell}(z))\\
&<\frac{\Delta c}{2}+\frac{\Delta c}{2}-\Delta c=0.
\end{aligned}
$$
Therefore $c$ is not increasing and the solution $(c,u)$ corresponding to $c_R'$ is of type I. It yields that
$$
\ell+\frac{\Delta \ell}{2}\in\{c'_R\ge 0: \text{The solution to (\ref{c_ODE_for_right})-(\ref{u_ODE_for_right}) is of Type I}\},
$$
which contradicts the definition of $\ell$. Hence the solution $(c_{\ell},u_{\ell})$ corresponding to $c_R'=\ell$ is of type II, which implies that $c_{\ell}$ is increasing.

 Now we show that the solution $(c_\ell,u_\ell)$ satisfies $c_{\ell}(+\infty)=c_B$. If not, since $c_{\ell}$ is increasing, we can define that $c(+\infty)=c_{\infty}\not=c_B$ ($c_{\infty}$ can take the value $+\infty$).  Since $c_{\ell}(R)>\bar{c}$ and  $c_{\ell}(x)$ is increasing on $[R,\infty)$, we have$$
\lim\limits_{x\rightarrow\infty} n_Re^{\int_{R}^{x}-\frac{G(c(y))}{\sigma}dy}=0,
$$
and therefore 
$$
\lim\limits_{x\rightarrow\infty} \psi(n_Re^{\int_{R}^{x}-\frac{G(c(y))}{\sigma}dy})=0.
$$
If $c_{\infty}<c_B$, we take $x\rightarrow\infty$ in equation (\ref{util_ODE_for_right}) and obtain $u_{\ell}(+\infty)=c_{\infty}-c_B<0$, which contradicts with the fact that $(c_{\ell},u_{\ell})$ is of type II. If $c_{\infty}>c_B$, we have $u_{\ell}(+\infty)=c_{\infty}-c_B>0$ by the same argument, and hence 
$$\inf_{x\ge R}u_{\ell}(x)>0.$$
By the continuous dependence of the solution to (\ref{problem_for_right}), we know that there exists $\Delta \ell>0$, such that when $|c_R'-\ell|<\Delta \ell$, the corresponding solution $u$ satisfies $u(x)>0$ in $[R,\infty)$, and hence the solution $(c,u)$ is of type II. It implies that
$$(\ell-\Delta \ell,\ell]\cap \{c'_R\ge 0: \text{The solution to (\ref{c_ODE_for_right})-(\ref{u_ODE_for_right}) is of Type I}\}=\emptyset,$$
which contradicts with the definition of $\ell$. 
 \paragraph{\textbf{Step 4. Conclude the proof by defining $B(\sigma,c_R)=\ell$.}}
We first prove the following proposition:
\begin{proposition}\label{monotone_prop}
given $c_R^1>(\ge)c_R^2$. Suppose that $(c_i,u_i)(i\in\{1,2\})$ are the solutions to the following systems:
$$
\begin{dcases}
-\sigma \partial_xn=nG(c),&x>R;\\
-\partial_{xx}c+\psi(n)c=(c_B-c),&x>R;\\
n(R+)=n_R,\,c(R)=c_R^i,\,\partial_xc(R)=c'^i_R.
\end{dcases}
$$
If $(c_1,u_1)$ and $(c_2,u_2)$ satisfy:
\begin{itemize}
	\item $c_i$ and $c_2$ are increasing;
	\item $c_1(+\infty)=c_2(+\infty)=c_B$,
\end{itemize}
then $c_R'^1<(\le) c_R'^2$.
\end{proposition}

\begin{proof}
We first prove the case when $c_R^1>c_R^2$. If the proposition does not hold, we have $c_R'^1\ge c_R'^2$. Then it is not hard to check that $c_1(R)>c_2(R),\,u_1(R)\ge u_2(R),\,\partial_{x}u_1(R)>\partial_{x}u_2(R)$ and $c_1(+\infty)=c_2(+\infty)$. By the facts that $u_1(R)\ge u_2(R)$ and $\partial_{x}u_1(R)>\partial_{x}u_2(R)$, we know that there exists $\delta>0$, such that when $y\in (R,R+\delta)$, $u_1(y)>u_2(y)$. By the facts that $c_1(R)>c_2(R)$ and $c_1(+\infty)=c_2(+\infty)$, we know that there exists $x>R$ such that $u_1(x)\le u_2(x)$.

Therefore, there exists $x_0>R$ such that $u_1(x_0)=u_2(x_0)$ and $u_1(y)>u_2(y)$ for all $R< y<x_0$. It follows that $\partial_{x}u_1(x_0)-\partial_{x}u_2(x_0)\le 0$ and $c_1-c_2$ is increasing on $[R,x_0]$. Denote $\Delta c=c_1(x_0)-c_2(x_0)> c_R^1-c_R^2>0$, we have
$$
\begin{aligned}
\partial_{x}u_1(x_0)-\partial_{x}u_2(x_0)&=[(c_1(x_0)-c_B)+\psi(n_Re^{\int_{R}^{x_0}-\frac{G(c_1(y))}{\sigma}dy}c_1]-[(c_2(x_0)-c_B)+\psi(n_Re^{\int_{R}^{x_0}-\frac{G(c_2(y))}{\sigma}dy}c_2]\\
&\ge\Delta c-c_2\|\psi'\|n_R(e^{\int_{R}^{x_0}-\frac{G(c_2(y))}{\sigma}dy}-e^{\int_{R}^{x_0}-\frac{G(c_1(y))}{\sigma}dy})\\
&\ge\Delta c-c_B\|\psi'\|e^{\int_{R}^{x_0}-\frac{G(c_2(y))}{\sigma}dy}(1-e^{-\frac{\|G'\|\Delta c}{\sigma}(R-x_0)})\\
&\ge \Delta c-c_B\|\psi'\|e^{-\frac{(R-x_0)G(\bar{c})}{\sigma}}\|G'\|\frac{\Delta c}{\sigma}(R-x_0)\\
&\ge \Delta c(1-\frac{c_B\|\psi'\|\|G'\|}{eG(\bar{c})})\\
&>0.
\end{aligned}
$$
It contradicts with the condition $\partial_{x}u_1(x_0)-\partial_{x}u_2(x_0)\le 0$.

Now we consider the case when $c_R^1\ge c_R^2$. If the proposition does not hold, we have $c_R'^1>c_R'^2$. Then it is not hard to check that $c_1(R)\ge c_2(R),\,u_1(R)> u_2(R),\,\partial_{x}u_1(R)\ge \partial_{x}u_2(R)$ and $c_1(+\infty)=c_2(+\infty)$. By similar arguments as the case $c_R^1>c_R^2$, we can still find $x_0>R$ such that $u_1(x_0)=u_2(x_0)$ and $u_1(y)>u_2(y)$ for all $R\le y<x_0$. Additionally, $\Delta c=c_1(x_0)-c_2(x_0)>c_R^1-c_R^2\ge0$. By the same argument as in the previous case, we can also reach a contradiction.
\end{proof}

The case $c_R^1>c_R^2$ in proposition (\ref{monotone_prop}) implies that $\ell$ is decreasing with respect to $c_R$. The case $c_R^1=c_R^2$ in proposition (\ref{monotone_prop}) implies that there does not exist $\tilde{\ell}\not=\ell$ such that the solution with $u(R)=\tilde{\ell}$ can meet all the requirements. Hence $B(\sigma,c_R):=\ell$ is well-defined and decreasing with respect to $c_R$. Now the continuity of $B(\sigma,c_R)=\ell$ with respect to $(\sigma,c_R)$ can be directly deduced by the fact that the solution to system (\ref{problem_for_right}) is continuously dependent on the parameter $\sigma$ and the initial data $c_R,c_R'$. We omit the details for simplicity.

\end{proof}

\section{Conclusion and discussion}\label{sec:conclusion}

In this paper we propose a free boundary tumor growth model motivated from the incompressible limit. The key feature is that at each time the pressure $p$ solves an obstacle problem on a time-evolving domain $\Omega(t)$, which gives the necrotic core as the coincidence set $\Lambda(t)$. This formulation is inspired by \cite{perthame2014traveling,guillen2020heleshaw} and is expected to be a special case of \cite{guillen2020heleshaw} under certain assumptions. 

We study two analytical aspects of the proposed model. For the formation of the necrotic core, we derive a semi-analytical solution, which illustrates the evolution of solution profiles from transition behavior to long term dynamics. For the effect of the outer density, we prove the existence of the traveling wave, which turns out that a smallness condition is needed.

In literature, free boundary models with a necrotic core have been proposed directly, rather than being derived from an incompressible limit, on which there are many studies including mathematical analysis and numerics \cite{byrne1996growth,cui2001analysis,cui2006formationNecrotic,HAO2012694,cui2019analysisNecrotic,wu2021analysis}. These models also feature an evolving domain $\Omega(t)$, coupled with the dynamics of pressure $p$ and nutrients $c$. Yet the specific mechanisms are different. In particular, instead of imposing $p=0$ on $\p\Omega(t)$,  a boundary condition involving the surface tension of $\p\Omega(t)$ is considered.

We want to highlight one interesting key distinction: how the necrotic core is characterized. Notably, these more classical free boundary models also have an obstacle problem formulation (see e.g. \cite[Between (1.6)-(1.7)]{cui2019analysisNecrotic}), but there it is the nutrient $c$ that solves an obstacle problem. In other words, the tumor region transitions to the necrotic core when the nutrient is below a critical level. Whereas in our model, it is the pressure $p$ that solves an obstacle problem, whose coincidence set gives the necrotic core $\Lambda(t)$. It is also interesting to observe that in Theorem \ref{thm:3.1}, the region where the nutrient is less than a critical threshold ($G(c)<0$), is strictly larger than $\Lambda(t)$ (where $p=0$). 

\section*{Acknowledgements}
ZZ is supported by the National Key R\&D Program of China, Project Number 2021YFA1001200, and the NSFC, grant Number 12031013, 12171013. We thank Jian-Guo Liu, Benoît Perthame and Jiajun Tong for helpful discussions.

\begin{appendices}
\addtocontents{toc}{\protect\setcounter{tocdepth}{0}}
\section{A Technical Lemma}
To prove the main theorem, we need a technical lemma introduced in \cite{perthame2014traveling}.
\begin{lemma}\label{lemma_for_left}
    Let $\sigma>0$ and $g_->0$ be given. Suppose that assumption \ref{assum_for_psi_G} holds. We consider the following system
    \begin{equation}\label{problem_for_left}
    \begin{cases}
    \sigma \partial_{x}n=ng_-,&x<0;\\
    -\partial_{xx}c+\psi(n)c=0,&x<0;\\
    c(0)=c_0,\,\partial_{x}c(0)=c_0'.
    \end{cases}
    \end{equation}
    Then, there exists a continuous mapping $\sigma\rightarrow A(\sigma)\in(0,\sqrt{\psi(1)}]$ such that the solution of the system (\ref{problem_for_left}) is non-negative and non-decreasing if and only if $c_0'=A(\sigma)c_0$. Additionally, $A(\cdot)$ can be extended by continuity to $A(0)=0$.
\end{lemma}
To prove this lemma, we refer to Lemma 3.1 in \cite{perthame2014traveling}, which provides the proof.
\end{appendices}
\bibliography{intro}

\begin{thebibliography}{10}

\bibitem{bubba2019-helelimit2}
F.~Bubba, B.~Perthame, C.~Pouchol, and M.~Schmidtchen.
\newblock Hele--shaw limit for a system of two reaction-(cross-) diffusion
  equations for living tissues.
\newblock {\em Archive for Rational Mechanics and Analysis}, pages 1--32, 2019.

\bibitem{bull2020mathematical}
J.~A. Bull, F.~Mech, T.~Quaiser, S.~L. Waters, and H.~M. Byrne.
\newblock Mathematical modelling reveals cellular dynamics within tumour
  spheroids.
\newblock {\em PLoS computational biology}, 16(8):e1007961, 2020.

\bibitem{byrne1996growth}
H.~M. Byrne and M.~Chaplain.
\newblock Growth of necrotic tumors in the presence and absence of inhibitors.
\newblock {\em Mathematical biosciences}, 135(2):187--216, 1996.

\bibitem{chaplain1993development}
M.~A. Chaplain.
\newblock The development of a spatial pattern in a model for cancer growth.
\newblock In {\em Experimental and Theoretical Advances in Biological Pattern
  Formation}, pages 45--59. Springer, 1993.

\bibitem{cui2006formationNecrotic}
S.~Cui.
\newblock Formation of necrotic cores in the growth of tumors: analytic
  results.
\newblock {\em Acta Mathematica Scientia}, 26(4):781--796, 2006.

\bibitem{cui2019analysisNecrotic}
S.~Cui.
\newblock Analysis of a free boundary problem modeling the growth of necrotic
  tumors.
\newblock {\em arXiv preprint arXiv:1902.04066}, 2019.

\bibitem{cui2001analysis}
S.~Cui and A.~Friedman.
\newblock Analysis of a mathematical model of the growth of necrotic tumors.
\newblock {\em Journal of Mathematical Analysis and Applications},
  255(2):636--677, 2001.

\bibitem{david:hal-02515263}
N.~David and B.~Perthame.
\newblock {Free boundary limit of tumor growth model with nutrient}.
\newblock working paper or preprint,hal-02515263, Mar. 2020.

\bibitem{dou2023tumor}
X.~Dou, J.-G. Liu, and Z.~Zhou.
\newblock A tumor growth model with autophagy: The reaction-(cross-) diffusion
  system and its free boundary limit.
\newblock {\em Discrete and Continuous Dynamical Systems-B}, 28(3):1964--1992,
  2023.

\bibitem{falcó2023quantifying}
C.~Falcó, D.~J. Cohen, J.~A. Carrillo, and R.~E. Baker.
\newblock Quantifying tissue growth, shape and collision via continuum models
  and bayesian inference, 2023.

\bibitem{feng2023tumor}
Y.~Feng, M.~Tang, X.~Xu, and Z.~Zhou.
\newblock Tumor boundary instability induced by nutrient consumption and
  supply.
\newblock {\em Zeitschrift f{\"u}r angewandte Mathematik und Physik},
  74(3):107, 2023.

\bibitem{fernandez2023regularity}
X.~Fern{\'a}ndez-Real and X.~Ros-Oton.
\newblock Regularity theory for elliptic pde.
\newblock {\em arXiv preprint arXiv:2301.01564}, 2023.

\bibitem{MR1009785}
A.~Friedman.
\newblock {\em Variational principles and free-boundary problems}.
\newblock Robert E. Krieger Publishing Co., Inc., Malabar, FL, second edition,
  1988.

\bibitem{friedman2007mathematical}
A.~Friedman.
\newblock Mathematical analysis and challenges arising from models of tumor
  growth.
\newblock {\em Mathematical Models and Methods in Applied Sciences},
  17(supp01):1751--1772, 2007.

\bibitem{guillen2020heleshaw}
N.~Guillen, I.~Kim, and A.~Mellet.
\newblock A hele-shaw limit without monotonicity.
\newblock {\em Archive for Rational Mechanics and Analysis}, 243(2):829--868,
  2022.

\bibitem{HAO2012694}
W.~Hao, J.~D. Hauenstein, B.~Hu, Y.~Liu, A.~J. Sommese, and Y.-T. Zhang.
\newblock Bifurcation for a free boundary problem modeling the growth of a
  tumor with a necrotic core.
\newblock {\em Nonlinear Analysis: Real World Applications}, 13(2):694 -- 709,
  2012.

\bibitem{jass1987new}
J.~Jass, S.~Love, and J.~Northover.
\newblock A new prognostic classification of rectal cancer.
\newblock {\em The Lancet}, 329(8545):1303--1306, 1987.

\bibitem{kienast2010real}
Y.~Kienast, L.~Von~Baumgarten, M.~Fuhrmann, W.~E. Klinkert, R.~Goldbrunner,
  J.~Herms, and F.~Winkler.
\newblock Real-time imaging reveals the single steps of brain metastasis
  formation.
\newblock {\em Nature medicine}, 16(1):116--122, 2010.

\bibitem{kim2018porous}
I.~Kim and N.~Po{\v{z}}{\'a}r.
\newblock Porous medium equation to hele-shaw flow with general initial
  density.
\newblock {\em Transactions of the American Mathematical Society},
  370(2):873--909, 2018.

\bibitem{lowengrub2010nonlinear}
J.~Lowengrub, H.~B. Frieboes, F.~Jin, Y.~Chuang, X.~Li, P.~Macklin, S.~M. Wise,
  and V.~Cristini.
\newblock Nonlinear modelling of cancer: bridging the gap between cells and
  tumours.
\newblock {\em Nonlinearity}, 23(1), 2010.

\bibitem{ahelemellet2017hele}
A.~Mellet, B.~Perthame, and F.~Quiros.
\newblock A hele--shaw problem for tumor growth.
\newblock {\em Journal of Functional Analysis}, 273(10):3061--3093, 2017.

\bibitem{perthame2015some}
B.~Perthame.
\newblock Some mathematical models of tumor growth.
\newblock 2015.
\newblock \url{https://www.ljll.math.upmc.fr/perthame/cours_M2.pdf}.

\bibitem{perthame2014heleAsym}
B.~Perthame, F.~Quir{\'o}s, and J.~L. V{\'a}zquez.
\newblock The hele--shaw asymptotics for mechanical models of tumor growth.
\newblock {\em Archive for Rational Mechanics and Analysis}, 212(1):93--127,
  2014.

\bibitem{perthame2014traveling}
B.~Perthame, M.~Tang, and N.~Vauchelet.
\newblock Traveling wave solution of the hele--shaw model of tumor growth with
  nutrient.
\newblock {\em Mathematical Models and Methods in Applied Sciences},
  24(13):2601--2626, 2014.

\bibitem{roose2007mathematical}
T.~Roose, S.~J. Chapman, and P.~K. Maini.
\newblock Mathematical models of avascular tumor growth.
\newblock {\em Siam Review}, 49(2):179--208, 2007.

\bibitem{wu2021analysis}
J.~Wu.
\newblock Analysis of a nonlinear necrotic tumor model with two free
  boundaries.
\newblock {\em Journal of Dynamics and Differential Equations}, 33(1):511--524,
  2021.

\end{thebibliography}
\bibliographystyle{abbrv}
\end{document}